\documentclass[a4paper,10pt]{article}

\usepackage{amssymb,amsfonts,amsmath,amsfonts,amsthm}
\usepackage[dvips]{graphicx,hyperref}
\usepackage{cite}
\usepackage{caption}
\newtheorem{theorem}{Theorem}[section]

\newtheorem{lemma}[theorem]{Lemma}

\theoremstyle{definition}

\newtheorem{remark}[theorem]{Remark}
\newtheorem{hypothesis}[theorem]{Hypothesis}
\newtheorem{corollary}[theorem]{Corollary}
\newcommand{\R}{\mathbb{R}}
\newcommand{\N}{\mathbb{N}}
\newcommand{\sym}{{\rm sym}}
\newcommand{\fz}{\frac}
\newcommand{\prz}[2]{ \frac{\partial{#1}}{\partial{#2}} }
\newcommand{\pz}{\partial}

\allowdisplaybreaks[4]

\title{
Stability estimates and a Lagrange--Galerkin scheme for a Navier--Stokes type model of flow in non-homogeneous porous media
}
\author{
Imam Wijaya\footnote{Corresponding author: {\tt robiah.wijaya239@gmail.com}}
\medskip\\
Division of Mathematical and Physics Science, Kanazawa University,\\
Kakuma, Kanazawa 920-1192, Japan.\\
\bigskip\\
Hirofumi Notsu
\medskip\\
Faculty of Mathematical and Physics, Kanazawa University,\\
Kakuma, Kanazawa 920-1192, Japan.
\smallskip\\
Japan Science and Technology Agency, PRESTO,\\
Kawaguchi 332--0012, Japan.
}

\date{}
\begin{document}
\maketitle
\begin{abstract}
The purposes of this work are to study the $L^{2}$-stability of a Navier--Stokes type model for non-stationary flow in porous media proposed by Hsu and Cheng in 1989 and to develop a Lagrange--Galerkin scheme with the Adams--Bashforth method to solve that model numerically.
The stability estimate is obtained thanks to the presence of a nonlinear drag force term in the model which corresponds to the Forchheimer term.
We derive the Lagrange--Galerkin scheme by extending the idea of the method of characteristics to overcome the difficulty which comes from the non-homogeneous porosity.
Numerical experiments are conducted to investigate the experimental order of convergence of the scheme.
For both simple and complex designs of porosities, our numerical simulations exhibit natural flow profiles which well describe the flow in non-homogeneous porous media.
\end{abstract}
%
%
%
%
%
%
%
\section{Introduction}
Fluid flow in porous media has received considerable attention in many kinds of applications such as in geophysics, petroleum engineering, and geothermal engineering, cf., e.g.,~\cite{1,4,13}.
In geothermal engineering, simulation of fluid flow and heat transfer in porous media is a useful tool not only for the pre-exploration process but also during the exploration process. For the pre-exploration process, the simulation can be used to predict how much electricity can be produced and determine how long the reservoir can be explored by using the physical parameters such as pressure, temperature, density, porosity, size of the reservoir, and the type of reservoir obtained from seismic data as an input parameter. From this simulation, we can determine the feasibility of a reservoir to be explored. During the exploration, simulation is used to predict the pressure and temperature changes in the reservoir because of injection and extraction processes. Injection is needed to maintain the balance of the mass in the reservoir and to supply the water which will be heated by the reservoir. In the extraction process, the fluid and steam are produced from the reservoir and used to generate electricity.
\par
For the underground flow, the so called Darcy law~\cite{4} is widely employed. However, the Darcy law is not appropriate in the geothermal application, since the porosity is non-homogeneous and the flow is non-stationary due to injection and extraction processes. 
\par
The analysis of fluid flow in porous media was started from H.~Darcy. In 1856 he observed the water flow in packed sand. His experiments were performed with a constant temperature single fluid and homogeneous porous media. According to his experiment, he concluded that the fluid velocity is proportional to pressure gradient. Then resulting Darcy equation in one dimensional case is 
\[
u = -k_D\frac{\partial p}{\partial x},
\]
where $u$ is the so called Darcy velocity, cf.~\eqref{da} below, $k_D$ is the Darcy permeability, $p$ is the pressure, and $x$ is the spatial coordinate.
To accommodate the thermal effect in Darcy's equation, A. Hazen~\cite{16} introduced the specific permeability $K$ and showed that the Darcy permeability is given by $k_D=\frac{K}{\mu}$, where $\mu$ is the temperature dependent dynamic viscosity. J.~Kozeny and P.C.~Carman gave a concrete form of the specific permeability $K$ in terms of the porosity $\phi$ and the particle diameter $d_p$ later.
\par
Darcy's law is the basic equation for modeling steady flow in porous media. This law assumes that the viscous forces dominate over inertial forces in porous media; hence, the inertial forces can be neglected. In the application where the permeability and porosity of the media are small such as in the groundwater and petroleum flows~\cite{4,13}, Darcy's law has an excellent performance to describe that phenomenon. However, in the application where the permeability and porosity of the medium are significantly large such as in the geothermal system, Darcy's law failed to describe it~\cite{5,17,24,12}.
\par
To improve Darcy's law, in 1947, H.C.~Brinkman added viscosity term which represents the shear stress term, and proposed the Brinkman equation~\cite{27}:      
\[
\frac{dp}{dx}=\mu\frac{\partial^{2}u}{\partial x^{2}}-\frac{\mu}{K}u.
\]
In the case of small porosity and permeability, the viscosity effect in pore throat is small, then the Brinkman equation is reduced to Darcy's law~\cite{24}. The Brinkman equation describes the transport processes in the porous media more generally than Darcy's equation. However, it only can be applied in a steady state.
\par
J.~Dupuit (1863) and P.~Forchheimer (1901) empirically found that as the flow rate increases, the inertial forces become significantly large, and the relationship between the pressure drop and velocity becomes nonlinear~\cite{24}.
With that fact, J.~Dupuit and P.~Forchheimer added a quadratic term of the velocity to represent the microscopic inertial effect, then resulting the Brinkman--Forchheimer equation:
\[
\frac{dp}{dx}=\mu\frac{\partial^{2}u}{\partial x^{2}}-\frac{\mu}{K}u-\beta\rho u^2,
\]
where $\beta=\frac{F\phi}{\sqrt{K}}$ is the non-Darcy coefficient, $F$ is the Forchheimer constant, $\phi$ is the porosity, and $\rho$ is the density of the fluid. This equation is more general than the Brinkman equation, but again, it is only applied in steady a state.
\par
S.~Whitaker (1967) introduced the volume average method to relates the volume average of a spatial derivative to the spatial derivative of the volume average, and makes the transformation from microscopic equations to macroscopic equations possible~\cite{12}.
C.T.~Hsu and P.~Cheng (1989) applied the volume average in the representative elementary volume (REV) to derive the equation for fluid flow in porous media. In their equation, they represented the drag force with Ergun's relation~\cite{9,7,15}. This approximation can be used to model the fluid flow in a geothermal reservoir for non-stationary condition.
\par
The purposes of this work are to study the $L^{2}$-stability of a Navier--Stokes type model for non-stationary flow in porous media proposed by C.T.~Hsu and P.~Cheng in 1989 and to develop a Lagrange--Galerkin scheme with the Adams--Bashforth method to solve that model numerically. A manufactured solution is employed to investigate the experimental order of convergence of the scheme in Subsection~\ref{subsec:eoc}.
To check the agreement of our simulation with the reality of fluid flow in porous media qualitatively, we set two cases of simulation and present the results in Subsection~\ref{subsec:nonhomo_porosity}.
\section{Governing equations}\label{sec:eqns}
C.T.~Hsu and P.~Cheng~\cite{7} reported the macroscopic continuity of mass and momentum equation for fluid flow in porous media based on the average of the microscopic continuity of mass and momentum over the REV. In this technique, the ``average theorems'' proposed by S.~Whitaker and J.C.~Slattery are needed to relate the average of the derivative to the derivative of average~\cite{9,11,15}. In this section, we will briefly review the ``average theorems'' for subsequent derivations.
\begin{figure}[htp]
	\centering
	\includegraphics[width=4.5in]{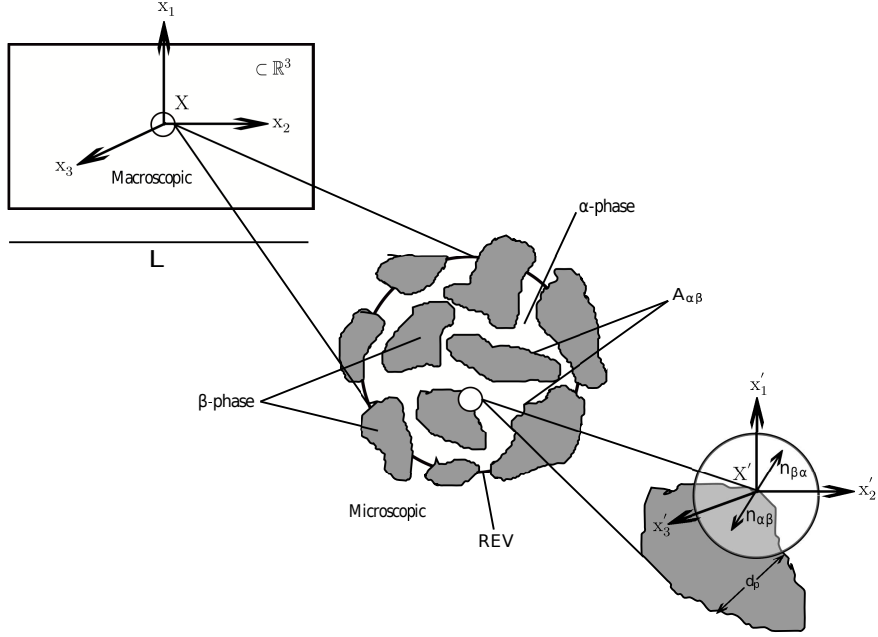}
	\caption{Representative elementary volume (REV)}
	\label{fig:REV}
\end{figure}
\par
Let us consider the porous media composed of $\alpha$ and $ \beta$ phases which represent fluid and solid, respectively. Let $\Omega\subset\mathbf{\R}^3$ be a bounded (macroscopic) domain. For $x\in\Omega$, let $V_{\alpha}(x)$ and $V_{\beta}(x)$ be microscopic volumes of $\alpha$ and $ \beta$ phases, respectively, and let $V(x) := V_{\alpha}(x) \cup V_{\beta}(x)\subset\mathbf{\R}^3$ be an REV satisfying $|V(x)|=|V_{\alpha}(x)|+|V_{\beta}(x)|<\infty$, where $|V_{\alpha}(x)|$ represents the measure of $V_{\alpha}(x)$. We assume $|V(x)|$ is constant and is denoted by $|V|$ and the porosity is given by $\phi(x)=\frac{|V_\alpha(x)|}{|V|}\in (0,1]$.
We denote by $v'=v'(x',x)\in \R^3$  the microscopic velocity at $ x'\in V_\alpha(x)$, where $x'$ is denotes the coordinates of $V_\alpha(x)$.  Then we introduce the macroscopic average velocity by averaging $v'$ over $V_\alpha(x)$:
\[
\langle v'\rangle=\frac{1}{|V_{\alpha}(x)|}\int_{V_{\alpha}(x)} v'(x',x)dx'.
\]
\par
The ``average theorems'' assume the total macroscopic source of the system at a point $x$ is equal to the total microscopic source to the system at a point $x'$ and total flux through the surface $A_{\alpha\beta}$, see Fig.~\ref{fig:REV}. Then this assumption yields 
\begin{equation}\label{5}
\nabla\cdot \biggl[ \frac{1}{|V|}\int_{V_{\alpha}}v'dx' \biggr] = \frac{1}{|V|} \int_{V_{\alpha}}\nabla'\cdot v'dx'+\frac{1}{|V|}\int_{A_{\alpha\beta}}v'\cdot n_{\beta\alpha} \, ds,
\end{equation}
where $n_{\beta\alpha}$ is the unit normal vector from $\beta$-phase to the $\alpha$-phase and $ds$ is the arc-length on the interface $A_{\alpha\beta}$.
In other words, we assume
\[
\nabla\cdot(\phi \langle v' \rangle) =\phi\langle \nabla'\cdot v'\rangle+\frac{1}{|V|}\int_{A_{\alpha\beta}}v'\cdot n_{\beta\alpha} \, ds.
\]
\par
For the time-dependent case, S.~Whitaker and J.C.~Slattery assumed the microscopic velocity $v'(x',x,t)$ and pressure $p(x',x,t)$ are governed by the Navier--Stokes equations in $V_\alpha(x)$, and derived its macroscopic equations in porous media  by taking the average in REV.
The ``average theorems'' assumption as given in~\eqref{5} yields  
\begin{align*}
\rho\left[\frac{\partial u}{\partial t}+(u\cdot\nabla)\frac{u}{\phi} \right] & = -\nabla p+\mu\Delta u + B(u, \phi),\\
\nabla\cdot u & = 0,
\end{align*}
where $u$ and $p$ are the superficial macroscopic velocity and pressure defined by 
\[
u(x,t):=\frac{1}{|V|}\int_{V_{\alpha}(x)}v'(x',x,t)dx',
\qquad
p(x,t):=\frac{1}{|V|}\int_{V_{\alpha}(x)}p'(x',x,t)dx'.
\]
We remak that these superficial quantities are represented by their macroscopic average $\langle v' \rangle$ and $\langle p' \rangle$ as follows:
\begin{equation}
\label{da}
u(x,t)=\phi(x)\langle v'(\cdot, x,t)\rangle,
\qquad
p(x,t)=\phi(x)\langle p'(\cdot, x,t)\rangle.
\end{equation}
The superficial velocity~$u$ is called the Darcy velocity.
The term $B(u,\phi)$ represents the total drag force from the micro pore structure per unit volume which satisfies with S.~Ergun expression~\cite{9}:
\begin{equation}
\label{def:B}
B(u,\phi) = B(u,\phi; \mu, \rho,d_{p}) := - \frac{\mu\phi u}{K(\phi)} - \rho\frac{F(\phi)\phi \left|u \right| u}{\sqrt{K(\phi)}},
\end{equation}
where $F:(0,1]\to (0, \infty)$ and $K:(0,1]\to (0, \infty]$ are functions defined by
\begin{align}
F(\phi) & := \frac{b}{\sqrt{a\phi^3}}, &
K(\phi) & := \frac{d_p^2 \phi^3}{a (1-\phi)^2},
\label{def:F_and_K}
\end{align}
which correspond to Forchheimer constant and Kozeny--Carman absolute permeability, respectively.
The constant~$d_p$ is a particle diameter, see Fig.~\ref{fig:REV}, and the values of $a$ and $b$ are empirically given by $a=150$ and $b=1.75$ in~\cite{11,15}.
\par
To clearly understand about the notation and the unit of our symbols, we summarized the units of important symbols in Table~\ref{tab:table1} below.
\begin{table}[htbp]
    \caption{The unit of important symbols}
    \label{tab:table1}
    \centering  
    \begin{tabular}{cccc} 
        \hline \hline                        
        No & Symbol  &  Unit  &  Name of the symbol \\ [0.5ex]
        \hline  
            1&$u$ & ${\rm m} \cdot {\rm s}^{-1}$ & Darcy velocity \\
            2&$p$ & ${\rm kg} \cdot {\rm m}^{-1}\cdot {\rm s}^{-2}$ & Pressure\\
            3&$\phi$ & -- & porosity \\
            4&$k_{D}$ & ${\rm kg}^{-1} \cdot {\rm m}^3 \cdot {\rm s}$ & Darcy permeability \\
            5&$K$ & ${\rm m}^2$ & Permeability\\
            6&$\mu$ & ${\rm kg} \cdot {\rm m}^{-1}\cdot {\rm s}^{-1}$ & Dynamic viscosity \\
            7&$\rho$ & ${\rm kg} \cdot {\rm m}^{-3}$ & Density\\
            8&$d_{p}$ & ${\rm m}$ & Particle diameter\\
            9&$F$ & -- & Forchheimer constant\\
            10&$B$ & ${\rm kg} \cdot {\rm m}^{-2} \cdot {\rm s}^{-2}$ & Drag force per unit volume\\
        \hline 
    \end{tabular}
\end{table}
%
%
%
\section{Statement of the problem}
In this section, we introduce a mathematical framework for the model presented in Section~\ref{sec:eqns}.
\par
The notation to be used in this paper is as follows.
For $d=2, 3$, let $\Omega \subset \R^d$ be a bounded domain, $\Gamma$ the boundary of $\Omega$, and $T$ a positive constant.
$\Gamma$ is divided into three parts, $\Gamma_i$, $i=0, 1, 2$, which satisfy $\bar{\Gamma} = \bar\Gamma_0 \cup \bar\Gamma_1 \cup \bar\Gamma_2$ and $\Gamma_i \cap \Gamma_j = \emptyset$ for all $i \neq j$.
We suppose that $\Gamma$ is a Lipschitz boundary, and that, for each $i\in \{0, 1, 2\}$, $\Gamma_i$ is piecewise smooth, where the total number of the smooth boundaries of~$\Gamma_i$ is finite.
The Lebesgue space on $\Omega$ for $p\in[1,\infty]$ is denoted by $L^p(\Omega)$ and the Sobolev space $W^{1,2}(\Omega)$ is denoted by $H^1(\Omega)$ with the norm
\[
\|u\|_{H^1(\Omega)}:=\left(\|u\|^2_{L^2(\Omega)} +\|\nabla u\|^2_{L^2(\Omega)}\right)^{1/2}. 
\]
The vector- and matrix-valued function spaces corresponding to, e.g.,~$L^2(\Omega)$ are denoted by $L^2(\Omega)^d$ and $L^2(\Omega)^{d\times d}$, respectively.
The inner products in~$L^2(\Omega)$, $L^2(\Omega)^d$, and $L^2(\Omega)^{d\times d}$ are all represented by $(\cdot , \cdot)$.
\par
We consider the following problem governed by the Navier--Stokes equations with non-homogeneous porosity~\cite{7};
find $(u, p): \overline{\Omega} \times [0,T]\to \R^d \times \R$ such that
\begin{subequations}\label{prob}
\begin{align}
\rho \Bigl[ \prz{u}{t} + (u\cdot\nabla) \frac{u}{\phi} \Bigr] - \nabla\cdot \left[ 2\mu D(u)\right]  +\nabla p & = f + B(u,\phi) && \text{in} \ \Omega\times(0,T), \\
\nabla\cdot u & = 0 && \text{in} \ \Omega\times(0,T), \\
u & = g && \text{on} \ \Gamma_{0}\times(0,T), \\
2\mu D(u) n-p n & = 0 && \text{on} \ \Gamma_{1}\times(0,T), \\
[2\mu D(u)n-p n] \times n & = 0 && \text{on} \ \Gamma_{2}\times(0,T), \\
u\cdot n & = 0 && \text{on} \ \Gamma_{2}\times(0,T), \\
u & = u^{0} && \text{in} \ \Omega,\ \text{at}\ t=0,
\end{align}
\end{subequations}
where
$u$ is the Darcy velocity,
$p$ is the pressure,
$\mu >0$ is a dynamic viscosity,
$u^0: \Omega \to \R^d$ is a given initial velocity,
$f: \Omega \times (0,T) \to \R^d$ is a given external force,
$g: \Gamma_0 \times (0,T) \to \R^d$ is a given boundary velocity,
$\phi: \Omega \to (0, 1]$ is a given porosity,
$D(u): \Omega \times (0, T) \to \R^{d \times d}_\sym$ is the strain-rate tensor defined by 
\[
D(u) := \frac{1}{2} \Bigl[ \nabla u + (\nabla u)^T \Bigr],
\]
$B(u, \phi) = B(u, \phi; \mu, \rho, d_{p}): \Omega \times (0, T) \to \R^d$ is the total drag force defined in~\eqref{def:B} with~\eqref{def:F_and_K}, and $n: \Gamma \to \R^d$ is the outward unit normal vector.
On the boundary, we impose the Dirichlet boundary condition on $\Gamma_0$, the stress free boundary condition on $\Gamma_1$, and the slip boundary condition on $\Gamma_2$.
\par
Throughout this paper, the following two hypotheses are assumed to hold.
\begin{hypothesis}\label{hyp1}
We suppose that ${\rm meas}(\Gamma_0) > 0$, 
$f \in C([0, T]; L^2(\Omega)^d)$, 
$g \in C([0, T]; H^1(\Omega)^d)$, and  
$u^0 \in L^2(\Omega)^d$.
\medskip
\end{hypothesis}
\begin{hypothesis}\label{hyp2}
	The porosity satisfies the following.
	\begin{itemize}
		\item[$(i)$]
		$\phi \in W^{1,\infty}(\Omega)$, \quad $\displaystyle\phi_0:=\mathop{\rm ess.inf}_{x\in\Omega }\phi(x)>0$.
		\item[$(ii)$]
		$\displaystyle|\nabla\phi|\leq\frac{2b}{d_p}(1-\phi)$ a.e. in $ \Omega$.
	\end{itemize}
\end{hypothesis}
Let us introduce constants~$\phi_1$ and~$\alpha$ defined by
\[
\phi_1:=\mathop{\rm ess.sup}_{x\in\Omega }\phi(x)\leq 1 ,\qquad  \alpha:=\frac{a(1-\phi_1)^2}{d_{p}^2\phi_{1}^2}\geq0.
\]
We note that
\begin{align}
\mathop{\rm ess.inf}_{x\in\Omega} \frac{\phi(x)}{K(\phi(x))} \ge \alpha \geq0.
\label{ieq:alpha}
\end{align}
%
%
%
\begin{remark}
From Hypothesis~\ref{hyp1} and the Trace Theorem~\cite{19}, it holds that $g(\cdot, t)_{|\Gamma_0} \in H^{1/2}(\Gamma_0)^d$ for any $t\in [0, T]$.
\end{remark}
\begin{remark}
As an example the value of $|\nabla \phi|$ in Lavrans field, Halten Terrace, Norway~\cite{28} is $4.336 \times 10^{-5}$ [cm$^{-1}$]. In the real situation, the value of $d_{p}\leq0.02$ [cm] and from the empirical study, S.~Ergun~\cite{9} suggested the value of $b=1.75$. Then if we calculate the right hand side term in Hypothesis~\ref{hyp2}-$(ii)$, it resulted $157.5$ [cm$^{-1}$]. Obviously, the spatial derivative of the real porosity $\nabla \phi(x)$ satisfies $|\nabla \phi| \ll 157.5$ [cm$^{-1}$]. By this fact, Hypothesis~\ref{hyp2}-$(ii)$ is not strict. 
\end{remark}
\par
For a function~$g_0\in H^{1/2}(\Gamma_0)^d$, let us introduce function spaces~$V(g_0)$, $V$, and~$Q$ defined by
\begin{align*}
V(g_0) := \bigl\{ v \in H^1(\Omega)^d;\ v=g_0 \ {\rm on} \ \Gamma_0,\ v\cdot n=0 \ {\rm on} \ \Gamma_2 \bigr\}, \  \, 
V := V(0), \  \,
Q := L^2(\Omega),
\end{align*}
respectively.
When $\Gamma=\Gamma_0$, we replace the definition of~$Q$ above with $Q:= L^2_0(\Omega) := \bigl\{ q\in L^2(\Omega);\; (q, 1) = 0 \bigr\}$
in a conventional way, cf.~\cite{19}.
We define bilinear forms $a_0$, $b$, and $c_0$, and trilinear forms $a_1$ and $c_1$ by
\begin{align*}
a_0(u, v) := 2\mu \bigl(  D(u), D(v) \bigr), \quad
b(v, q) := & -(\nabla\cdot v, q), \quad
c_0(u, v) := \mu \Bigl( \fz{\phi}{K(\phi)}u, v \Bigr), \\
a_1(u, w, v) := \rho \bigl(  (u\cdot\nabla) w, v \bigr), \quad
& c_1(\theta, u, v) := \rho \biggl( \fz{F(\phi) \phi \, \theta u}{\sqrt{K(\phi)}}, v \biggr).
\end{align*}
The weak formulation for problem~\eqref{prob} is to find $\{ (u, p) (t) \in V(g(t)) \times Q;\ t \in (0, T)\}$ such that, for $t\in (0,T)$,
\begin{subequations}\label{prob_weak}
\begin{align}
\rho \Bigl(\frac{\partial u}{\partial t}, v\Bigr) + a_0(u, v) + a_1\Bigl(u, \fz{u}{\phi}, v\Bigr) + b(v, p) & + b(u, q) + c_0(u, v) + c_1\bigl(\lvert u \rvert, u, v\bigr) \notag\\
& = \left( f(t),v\right), \quad
\forall(v, q) \in V \times Q, \\
u(0) & = u^{0} \quad \mbox{in}\ L^2(\Omega)^d.
\end{align}
\end{subequations}
%
\section{Stability estimates}
In this section, we present theoretical results, Theorem~\ref{thm} and Corollary~\ref{cor}, which provide a key inequality and stability estimates, respectively.
The stability estimates are easily derived from the key inequality.
\begin{theorem}\label{thm}
	Suppose that Hypotheses~\ref{hyp1} and~\ref{hyp2} hold true.
	Assume $g=0$.
	Suppose that $(u,p) \in (C^1([0,T];L^2(\Omega)^d) \cap L^2(0,T;V))\times L^2(0,T;L^2(\Omega))$ satisfies~\eqref{prob_weak}.
Then, it holds that
\begin{align}
\fz{d}{dt}\Bigl(\frac{\rho}{2}\lVert u(t) \rVert_{L^2(\Omega)}^{2} \Bigr) + \fz{\rho}{2} \int_{\Gamma_1} \fz{|u(t)|^2}{\phi} u(t)\cdot n\, ds + \mu \beta_0^2\| u(t) \|_{H^1(\Omega)}^2 + \mu\alpha \|u(t)\|_{L^2(\Omega)}^2 
\le \fz{1}{4\mu\beta_0^2} \|f(t)\|_{L^2(\Omega)}^2,
\label{ieq:thm}
\end{align}
where $\beta_0 > 0$ is a positive constant to be defined in~\eqref{ieq:Korn} below.
\end{theorem}
\begin{corollary}[Stability estimates]\label{cor}
In addition to the same assumptions in Theorem~\ref{thm}, suppose that ${u \cdot n \ge 0}$ on $\Gamma_1\times [0, T]$.
Then, we have the following.
\begin{itemize}
\item[$(i)$]
It holds that
\begin{align}
\sqrt{\rho}\lVert u\rVert_{L^{\infty}(0,T;L^2(\Omega))} & + \sqrt{\mu} \beta_0 \lVert u\rVert_{L^2(0,T;H^1(\Omega))}  \notag\\
& \le 2 \Bigl( \sqrt{\rho} \lVert u^0 \rVert_{L^2(\Omega)} + \fz{1}{\sqrt{\mu}\beta_0} \lVert f \rVert_{L^{2}(0,T;L^2(\Omega))} \Bigr).
\label{ieq:cor_1}
\end{align}
\item[$(ii)$]
It holds that, for any $t \in [0, T]$,
\begin{align}
\lVert u(t) \rVert_{L^2(\Omega)}
\le \exp\Bigl( -\fz{\mu\alpha}{\rho} \, t \Bigr) \lVert u^0 \rVert_{L^2(\Omega)} + \fz{1}{\sqrt{2\rho\mu}\beta_0} \lVert f \rVert_{L^{2}(0,t;L^2(\Omega))}.
\label{ieq:cor_2}
\end{align}
\end{itemize}
\end{corollary}
\par
The proofs of Theorem~\ref{thm} and Corollary~\ref{cor} are given after preparing two lemmas.
\begin{lemma}[Korn's inequality, \cite{20,21}]\label{lem:Korn}
Let $\Omega$ be a bounded domain with a {\rm Lipschitz}-continuous boundary~$\pz\Omega$, and let $\Gamma_0$ be a part of $\pz\Omega$ and piecewise {\rm Lipschitz}-continuous. Assume ${\rm meas}(\Gamma_0) > 0$.
Then, there exists a positive constant~$\beta_0$ such that
\begin{align}
\beta_0 \|u\|_{H^1(\Omega)} \le \|D(u)\|_{L^2(\Omega)}, \qquad \forall u \in \{ v \in H^1(\Omega)^d;\ v = 0\ \mbox{on}\ \Gamma_0\}.
\label{ieq:Korn}
\end{align}
\end{lemma}
\begin{lemma}\label{lem:a1}
	Suppose Hypothesis~\ref{hyp2}-$(i)$ holds true.
	Assume $u \in H^1(\Omega)^d$ and $\nabla\cdot u = 0$ in $\Omega$.
	Then, it holds that 
	\begin{equation}\label{identity_a1}
	\Bigl( (u\cdot\nabla)\Bigl(\frac{u}{\phi}\Bigr), u \Bigr) = \fz{1}{2} \int_\Gamma \fz{|u|^2}{\phi} u\cdot n\, ds + \fz{1}{2} \Bigl(  |u|^{2}, (u\cdot\nabla)\frac{1}{\phi} \Bigr).
	\end{equation}
\end{lemma}
\begin{proof}
	Let $I \equiv ( (u\cdot\nabla)(u/\phi), u)$.
	From the integration by parts, and the assumption, $\nabla \cdot u = 0$, the following identity holds:
	\begin{equation}\label{identity_I_1}
	I = \int_\Gamma \frac{|u|^2}{\phi} u \cdot n \, ds - \Bigl( \nabla \cdot (u \otimes u), \frac{u}{\phi} \Bigr) = \int_\Gamma \frac{|u|^2}{\phi} u \cdot n \, ds - \Bigl( (u\cdot\nabla)u, \frac{u}{\phi} \Bigr).
	\end{equation}
	On the other hand, from the product rule, we get another identity:
	\begin{align}
	I & = \Bigl( [(u\cdot\nabla)u] \frac{1}{\phi} + \Bigl[ (u\cdot\nabla)\Bigl( \frac{1}{\phi} \Bigr) \Bigr] u, u \Bigr) 
	 = \Bigl( (u\cdot\nabla)u, \frac{u}{\phi} \Bigr) + \Bigl( |u|^2, (u\cdot\nabla) \frac{1}{\phi} \Bigr).
	 \label{identity_I_2}
	\end{align}
	Adding the two equations~\eqref{identity_I_1} and~\eqref{identity_I_2} and dividing it by~$2$, we obtain~\eqref{identity_a1}.
\end{proof}
\begin{proof}[Proof of Theorem~\ref{thm}]
Substituting $(u, -p) \in V \times Q$ into $(v,q)$ in~\eqref{prob_weak}, we have
\begin{equation}\label{proof:eq1}
\rho\Bigl(\prz{u}{t}, u \Bigr) +a_0(u, u) +a_1\Bigl(u, \fz{u}{\phi}, u\Bigr) + c_0(u, u)+c_1(\lvert u\lvert, u, u) \le (f, u).
\end{equation}
We evaluate each term in~\eqref{proof:eq1} as follows:
\begin{subequations}\label{proof:estimates}
\begin{align}
\rho\Bigl(\prz{u}{t}, u \Bigr) & = \fz{d}{dt}\Bigl( \fz{\rho}{2}\lVert u \rVert_{L^2(\Omega)}^{2} \Bigr), \\
a_0(u, u) & = 2\mu \| D(u) \|_{L^2(\Omega)}^2 \ge 2\mu\beta_0^2 \| u \|_{H^1(\Omega)}^2 & \mbox{(by Lem.~\ref{lem:Korn})}, \\
a_1\Bigl(u, \fz{u}{\phi}, u\Bigr) & = \fz{\rho}{2} \int_{\Gamma_1} \fz{|u|^2}{\phi} u\cdot n\, ds + \fz{\rho}{2} \Bigl(  |u|^{2}, (u\cdot\nabla)\frac{1}{\phi} \Bigr) & \mbox{(by Lem.~\ref{lem:a1})} \, \notag\\
& \ge \fz{\rho}{2} \int_{\Gamma_1} \fz{|u|^2}{\phi} u\cdot n\, ds - \Bigl(  |u|^{2}, \fz{\rho |u|}{2} \Bigl|\nabla\frac{1}{\phi} \Bigr| \Bigr), \\
c_0(u, u) & = \mu\Bigl( \fz{\phi}{K(\phi)}, |u|^{2} \Bigr) \ge \mu\alpha \|u\|_{L^2(\Omega)}^2 & \mbox{(by~\eqref{ieq:alpha})}, \\
c_1(\lvert u\lvert, u, u) & = \biggl( |u|^2, \rho|u| \fz{F(\phi)\phi}{\sqrt{K(\phi)}} \biggr), \\
(f, u) & \le \mu\beta_0^2 \|u\|_{L^2(\Omega)}^2 + \fz{1}{4\mu\beta_0^2} \|f\|_{L^2(\Omega)}^2 \notag\\
& \le \mu\beta_0^2 \|u\|_{H^1(\Omega)}^2 + \fz{1}{4\mu\beta_0^2} \|f\|_{L^2(\Omega)}^2.
\end{align}
\end{subequations}
Here, we note the fact that Hypothesis~\ref{hyp2} yields
\begin{align}
G_\phi & := \fz{1}{2}\biggl| \nabla\fz{1}{\phi}\biggr| - \fz{F(\phi)\phi}{\sqrt{K(\phi)}} 
= \fz{1}{2\phi^2} \Bigl[ | \nabla\phi | - \fz{2b}{d_p} ( 1-\phi ) \Bigr]
\le 0
\quad \mbox{a.e. in}\ \Omega.
\label{ieq:G}
\end{align}
Combining~\eqref{proof:estimates} with~\eqref{proof:eq1} and using~\eqref{ieq:G}, we obtain
\begin{align*}
\fz{d}{dt}\Bigl(\frac{\rho}{2}\lVert u(t) \rVert_{L^2(\Omega)}^{2} \Bigr) + \fz{\rho}{2} \int_{\Gamma_1} \fz{|u(t)|^2}{\phi} u(t) \cdot n\, ds + \mu \beta_0^2\| u(t) \|_{H^1(\Omega)}^2 + \mu\alpha \|u(t)\|_{L^2(\Omega)}^2 \\
\le \fz{1}{4\mu\beta_0^2} \|f(t)\|_{L^2(\Omega)}^2 + \bigl( |u(t)|^{2}, \rho |u(t)| G_\phi \bigr)
\le \fz{1}{4\mu\beta_0^2} \|f(t)\|_{L^2(\Omega)}^2.
\end{align*}
Thus, we obtain~\eqref{ieq:thm}.
\end{proof}
\begin{proof}[Proof of Corollary~\ref{cor}]
Firstly, we prove~$(i)$.
Dropping the non-negative second and forth terms in~\eqref{ieq:thm}, we have
\[
\fz{d}{dt}\Bigl(\frac{\rho}{2} \lVert u (t) \rVert_{L^2(\Omega)}^{2} \Bigr) + \mu \beta_0^2\| u(t) \|_{H^1(\Omega)}^2 \le \fz{1}{4\mu\beta_0^2} \|f(t)\|_{L^2(\Omega)}^2,
\]
which implies~\eqref{ieq:cor_1}.
Here, we have used the fact that, for non-negative functions~$\eta\in C^1([0,T]; \R)$ and $\phi, \psi \in L^1([0,T]; \R)$, the inequality~$\eta^\prime(t) + \phi(t) \le \psi(t)$ $(t\in [0,T])$ yields $\|\eta\|_{L^\infty(0,T)} + \|\phi\|_{L^1(0,T)} \le 2 [\eta(0) + \|\psi\|_{L^1(0,T)}]$, and an inequality~$(a+b)/\sqrt{2} \le \sqrt{a^2+b^2}$~$(a, b \in\R)$.
\par
Secondly, we prove~$(ii)$.
Dropping the non-negative second and third terms in~\eqref{ieq:thm}, we get
\[
\fz{d}{dt}\Bigl(\frac{\rho}{2}\lVert u(t) \rVert_{L^2(\Omega)}^{2} \Bigr) + \mu\alpha \|u(t)\|_{L^2(\Omega)}^2 \le \fz{1}{4\mu\beta_0^2} \|f(t)\|_{L^2(\Omega)}^2,
\]
which implies~\eqref{ieq:cor_2} from Gronwall's inequality.
\end{proof}
%
\section{A Lagrange--Galerkin scheme}

In this section, we present a Lagrange--Galerkin scheme of second-order in time for problem~\eqref{prob}.
\par
For the Darcy velocity $u$ and the porosity $\phi$ in problem~\eqref{prob}, we introduce the macroscopic average velocity~$w: \overline{\Omega} \times [0,T] \to \R^d$ and the material derivative~$D/Dt$ with respect to~$w$ defined by
\begin{align*}
w & := \fz{u}{\phi}, & \fz{D}{Dt} & := \prz{}{t} + w\cdot\nabla.
\end{align*}
Then, we can rewrite $\pz u/\pz t + (u\cdot\nabla) (u/\phi)$ by
\begin{align}
\prz{u}{t}+(u\cdot\nabla)\fz{u}{\phi} = \phi \Bigl[ \prz{w}{t}+(w\cdot\nabla) w \Bigr] = \phi \fz{Dw}{Dt}.
\label{eq:mat_der}
\end{align}
The equation~\eqref{eq:mat_der} is a fundamental relation to the development of our new numerical scheme to be presented.
\par
Let $\tau$ be a time increment, $N_{T} := \lfloor T/\tau \rfloor$ the total number of time steps, and $t^k := k \tau$ for $k \in \{0, 1, \ldots, N_T\}$.
For a function~$\psi$ defined in~$\overline{\Omega}\times [0,T]$ or $\Gamma_0\times [0,T]$, we denote $\psi(\cdot, t^k)$ simply by $\psi^k$.
Let $X: [0,T] \to \R^d$ be a solution of the following ordinary differential equation,
\begin{equation}\label{ode}
X^\prime(t) = w(X(t),t), \quad t \in [0,T],
\end{equation}
subjected to an initial condition $X(t^k) = x$.
Physically, $X(t)$ represents the position of a fluid particle with respect to the macroscopic average velocity~$w$ at time~$t$.
For a given velocity $v: \Omega \to \R^d$, let $ X_1(v, \tau): \Omega\to\R^d$ be the mapping defined by 
\begin{equation}
X_1(v, \tau)(x):= x-v(x)\tau,
\label{def:X1}
\end{equation}
which is an upwind point of $x$ with respect to the velocity~$v$ and a time increment~$\tau$.
Now, we derive the second-order approximation of~$\pz u/\pz t + (u\cdot\nabla) (u/\phi)$ at $(x,t^k)$ by the Adams--Bashforth method as follows:
\begin{align}
& \Bigl[ \prz{u}{t}+(u\cdot\nabla)\fz{u}{\phi} \Bigr] (x,t^k) = \phi(x)\frac{Dw}{Dt}(x, t^k) = \phi(x) \fz{d}{dt}\left( w(X(t),t)\right)_{|t=t^k} \notag\\
& = \fz{\phi(x)}{2\tau} \Bigl[ 3w^k-4w^{k-1} \circ X_1\bigl(w^k, \tau\bigr) + w^{k-2} \circ X_1\bigl(w^k, 2\tau\bigr) \Bigr] (x) +O(\tau^2) \label{eq:derivation_md_AB}\\
& = \fz{\phi(x)}{2\tau} \Bigl[ 3w^k-4w^{k-1} \circ X_1\bigl(w^{(k-1)\ast}, \tau\bigr) + w^{k-2} \circ X_1\bigl(w^{(k-1)\ast}, 2\tau\bigr) \Bigr] (x) +O(\tau^2) \notag\\
&=\fz{1}{2\tau} \Bigl[ 3u^k - \phi \bigl[ 4 w^{k-1}\circ X_1(w^{(k-1)\ast},\tau) - w^{k-2}\circ X_1(w^{(k-1)\ast},2\tau)\bigr] \Bigr](x) +O(\tau^2),\notag
\end{align}
\noindent
where the symbol ``$\circ$'' denotes the composition of functions,
\[
[v\circ X_1(v,\tau)] (x) = v(X_1(v,\tau)(x)),
\]
and $w^{(k-1)\ast}$ is a second-order approximation of~$w^k$ defined by
\[
w^{(k-1)\ast} := 2 w^{k-1} - w^{k-2}.
\]
The idea of~\eqref{eq:derivation_md_AB} has been proposed and employed in~\cite{EwiRus-1981,BouMadMetRaz-1997,22,23}.
\par
Let $\mathcal{T}_{h}:= \{e\}$ be a triangulation of $\overline{\Omega}\left(=\cup_{e\in\mathcal{T}_{h}} \right) $, $h_e$ the diameter of $e \in\mathcal{T}_{h}$, and $h := \max_{e \in \mathcal{T}_h} h_e$ the maximum element size. 
We define the function spaces $X_h, M_h,V_h$ and $Q_h$ by 
\begin{align*}
X_h & := \bigl\{ v_{h}\in C(\overline\Omega)^d;\; v_{h|e} \in P_{2}(e)^d, \ \forall e\in\mathcal{T}_{h} \bigr\}, \\
M_h & := \bigl\{ q_{h}\in C(\overline\Omega);\; q_{h|e} \in P_{1}(e), \ \forall e\in\mathcal{T}_{h}\bigr\},
\end{align*}
$V_{h}:= X_{h}\cap V$, and $ Q_{h}:= M_{h}\cap Q=M_{h}$, respectively, where $P_{k}(e)$ is the (scalar-valued) polynomial space of degree $k\in\N$ on $e$.
\par
Let $u_h^0 \in X_h$ and $\{g_h^k\}_{k=1}^{N_T} \subset X_h$, approximations of~$u^0$ and~$g$, be given.
Our new Lagrange--Galerkin scheme of second-order in time for solving problem~\eqref{prob} is to find $\left\lbrace (u_{h}^{k},p_{h}^{k})\right\rbrace^{N_{T}}_{k=1} \subset V_{h}(g_h^k) \times Q_{h} $ such that, for all $(v_{h},q_{h}) \in V_{h} \times Q_{h}$,
\begin{subequations}\label{scheme}
\begin{align}
\intertext{(initial step)}
\biggl( \fz{u_h^1-\phi [w_h^0\circ X_1(w_h^0, \tau)]}{\tau}, v_h \biggr) + a_0(u_h^1,v_h) & + b(v_h, p_h^1) + b(u_h^1, q_h) \notag\\
+ c_0 (u_h^1, v_h) & + c_1(|u_h^0|, u_h^1, v_h) = (f^1, v_h),
\label{scheme_eq1}
\intertext{(general step)}
\biggl( \fz{1}{2\tau} \Bigl[ 3u_h^k - \phi \bigl[ 4 w_h^{k-1}\circ X_1(w_h^{(k-1)\ast},\tau) & - w_h^{k-2}\circ X_1(w_h^{(k-1)\ast},2\tau)\bigr] \Bigr], v_h \biggr) \notag\\
+ a_0(u_h^k, v_h)+b(v_h, p_h^k)+b(u_h^k,q_h) & + c_0(u_h^k, v_h) + c_1(|u_h^{(k-1)\ast}|, u_h^k, v_h,) \notag\\
& = (f^k, v_h),
\quad k=2,\ldots, N_{T}, 
\label{scheme_eq2}
\end{align}
\end{subequations}
where $w_h^k$ and $w_h^{(k-1)\ast}$ are defined by
\begin{align*}
w_h^k := \fz{u_h^k}{\phi}, \qquad
w_h^{(k-1)\ast} := 2w_h^{k-1} - w_h^{k-2}.
\end{align*}
We compute $(u_h^1, p_h^1)$ by~\eqref{scheme_eq1} and $\{(u_h^k, p_h^k)\}_{k=2}^{N_T}$ by~\eqref{scheme_eq2}.
This idea on the initial step treatment has been proposed for the Navier--Stokes equations, cf.~\cite{23}, where the second-order convergence in time in $L^2(\Omega)$-norm has been proved.
Here, we apply it to problem~\eqref{prob}.
%
\section{Numerical results}
In this section, we confirm the experimental order of convergence of scheme~\eqref{scheme} and perform some numerical simulation for fluid flow in non-homogeneous porous media.
\subsection{Order of Convergence}\label{subsec:eoc}
In this subsection, a two-dimensional test problem is computed by scheme~\eqref{scheme} to check the order of convergence of the scheme. In problem~\eqref{prob} we set $\Omega=(0,\pi)^{2}$~[cm], $T=1$~[s], $\mu=8.89\times10^{-3}$~[dyn$\cdot$s/cm$^2$],  $d_{p}=5\times 10^{-2}$~[cm], $\rho=9.951\times 10^{-1}$~[gr/cm$^3$], and $\phi=\left[ 2+\sin(\frac{2 y}{5})\right] /3$. The functions $g$ and $u^{0}$ are given so that the manufactured solution is
\[
u(x,t) = \Bigl( -\frac{\partial \psi}{\partial y},\frac{\partial \psi}{\partial x} \Bigr)(x,t),
\quad
p=\sin(x)\sin(y)e^{-2t},
\quad
\psi = \sin^3(x)\sin^3(y)e^{-2t}.
\]
\par
The problem is solved by scheme~\eqref{scheme} with $h=\pi/N$ for $N=4,8,16,32,128$, and $\tau=h$. 
For the computation we employed FreeFem++~\cite{26} with P2/P1-element. For the solution $(u_{h},p_{h})$ of scheme~\eqref{scheme} we define errors $Er1$ and $Er2$ by
\[
Er1 := \max_{n=0, \dots ,N_{T}}\parallel u_{h}^{n}-u^{n}\parallel_{H^{1}(\Omega)},
\quad
Er2 := \max_{n=0, \dots ,N_{T}}\parallel p_{h}^{n}-p^{n}\parallel_{L^{2}(\Omega)}.
\]
Figure~\ref{fig:EOC} shows the graphs of $Er1$ and $Er2$ versus $h~(=\tau)$ in logarithmic scale. The values of $Er1$, $Er2$ and slopes are represented in Table~\ref{tab:PPer}.
We can see that both $Er1$ and $Er2$ are almost of second order in $h~(=\tau)$. 
\begin{figure}[htbp]
	\begin{center}
		\includegraphics[width=3.0in]{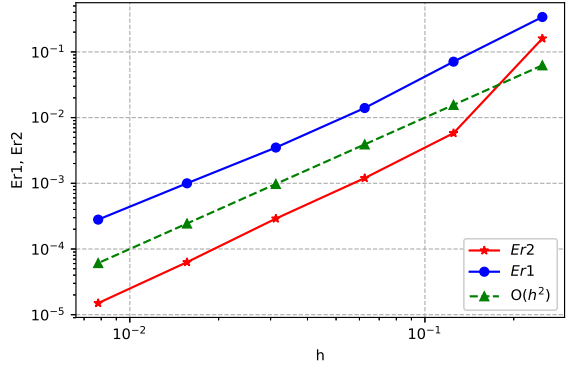}
		\caption{The order of convergence for scheme~\eqref{scheme}.}
		\label{fig:EOC}
	\end{center}
\end{figure}
\begin{table}[htbp]
	\caption{Values of $Er1$ and $Er2$ and their slopes for the problem in Subsection~\ref{subsec:eoc} by scheme~\eqref{scheme}.}
	\centering
	\begin{tabular}{l c c c c}
		\hline\hline
	$N$ & $Er1$ & $Er2$ & Slope of $Er1$ & Slope of $Er2$ \\ 
		\hline
		4& $3.4\times10^{-1}$ & $1.6\times10^{-1}$ & $-$ & $-$  \\
		8& $7.1\times10^{-2}$ & $5.8\times10^{-3}$ &2.2 & 4.8  \\
		16& $1.4\times10^{-2}$ &$1.2\times10^{-3}$ & 2.3 & 2.3   \\
		32& $3.5\times10^{-3}$ & $2.9\times10^{-4}$ & 2.0 & 2.0 \\ 
		64& $1.0\times10^{-3}$ & $6.3\times10^{-5}$ & 1.80 & 2.2 \\
		128& $2.8\times10^{-4}$ & $1.5\times10^{-5}$ & 1.84 & 2.1 \\
		\hline
	\end{tabular}
	\label{tab:PPer}
\end{table}
\subsection{Simulation with non-homogeneous porosity}\label{subsec:nonhomo_porosity}
In this subsection, we present two cases of numerical simulation for the fluid flow through the non-homogeneous porous media. 
\par
The purpose of the first case simulation is to understand the fluid flow in the two layers of porosity. This simulation motivated by the real condition of the geothermal reservoir which has porosity function of the depth. In the top of the reservoir, the value of porosity is large, while in the bottom, the value of porosity is small due to the existence of pressure which comes from the mass of the soils and rocks.
\par
We set $\Omega=(0,3)\times(0,1)$~[cm], $\Gamma_{1}= \left\lbrace (x_{1},x_{2});\ x_{1}=3, \, 0<x_{2}<1 \right\rbrace $, $ \Gamma_{0}=\partial\Omega / \overline{\Gamma}_{1}$, $f=0$, $g=u^{0}$ on $\Gamma_{0}$, $\Gamma_{2}=\emptyset$,  $T=5$ [s], $\rho=9.951\times 10^{-1}$~[gr/cm$^{3}]$, and $\mu=8.89\times 10^{-3}$~[dyn$\cdot$s/cm$^2]$.
We define the initial condition as  
\[
u^{0}=\eta(x_{1})
\begin{pmatrix}
	\frac{1}{4}-(x_{2}-\frac{1}{2})^{2}\\
	0
\end{pmatrix},
\]
where $\eta$ is defined by
\begin{equation}\label{eq:eta}
\eta(x_{1}) := 
	\left\{
	\begin{aligned}
	& \cos(\pi x_1)  &&  (0\leq x_{1}\leq 0.5), \\
	& \ 0 && (0.5< x_{1}).
	\end{aligned}
	\right.
\end{equation}
For the porosity $\phi$ we set
\begin{equation*}
	\phi(x)=0.4+0.4H_{\epsilon}(x_{2}-0.5),
\end{equation*}
where $\epsilon = \fz{1}{360}$ and $H_{\epsilon}$ is an approximate Heaviside function defined by
\[
	H_{\epsilon}(s)= 
	\left\{
	\begin{aligned} 
		& 1  &&  (s\geq\epsilon) , \\
		& \frac{1}{2}+\frac{1}{2}\left(\frac{s}{\epsilon}+\frac{1}{\pi}\sin\frac{\pi s}{\epsilon}\right) && (\left| s \right| < \epsilon), \\
		& 0  &&  (s\leq-\epsilon).
	\end{aligned}
	\right.
\]
For this case we run the simulation with division number $N=120$, $h=3/N$, $\tau=h$.
Since we have a layer of $\phi$ on $x_{2}=1/2$, we employ a mesh whose mesh size near $x_{2}=1/2$ is chosen as around $1/720$.
To aid the understanding of the problem setting in this simulation, the boundary conditions and the porosity are illustrated together with the finite element mesh on $\Omega$ in Figure~\ref{fig:mesh}.
\begin{figure}[htbp]
	\centering
	\includegraphics[width=3.5in]{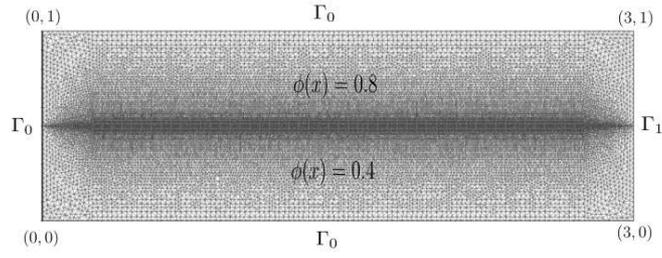}
	\caption{The boundary conditions and the finite element mesh.}\label{fig:mesh}
\end{figure}
\begin{figure}[htbp]
	\begin{minipage}[b]{0.49\linewidth}
		\centering
		\includegraphics[width=1.\linewidth]{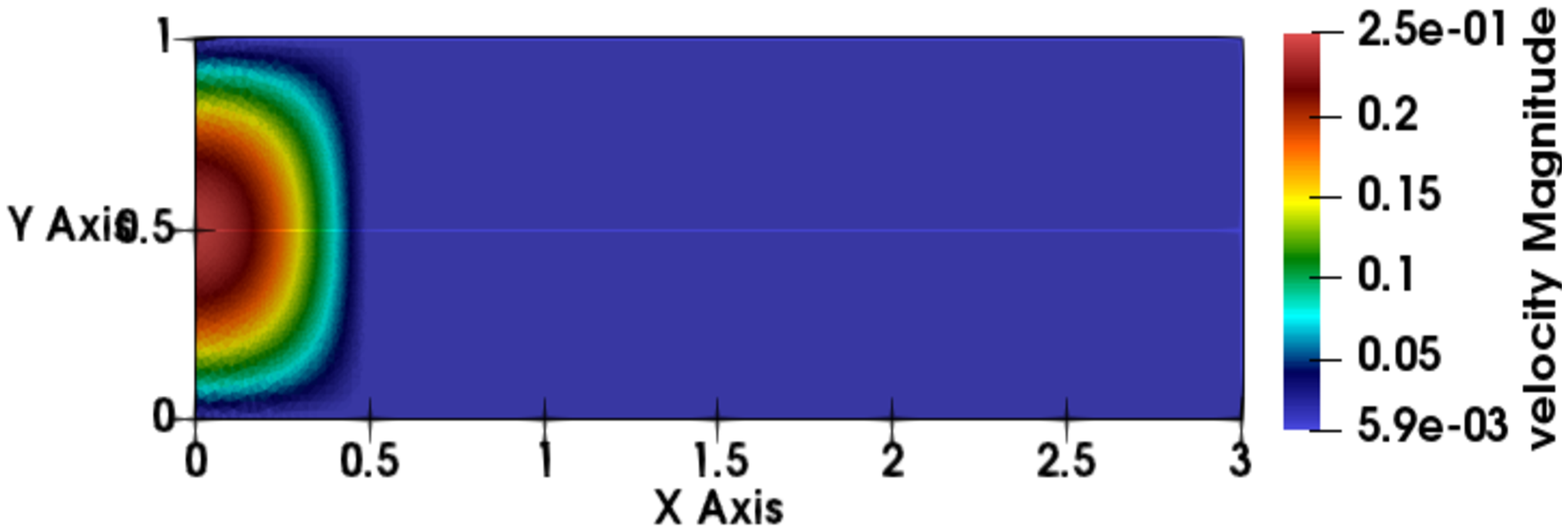}
		{\text{(a) $t=0.0$} [s]}
	\end{minipage}
	\begin{minipage}[b]{0.49\linewidth}
		\centering
		\includegraphics[width=1.\linewidth]{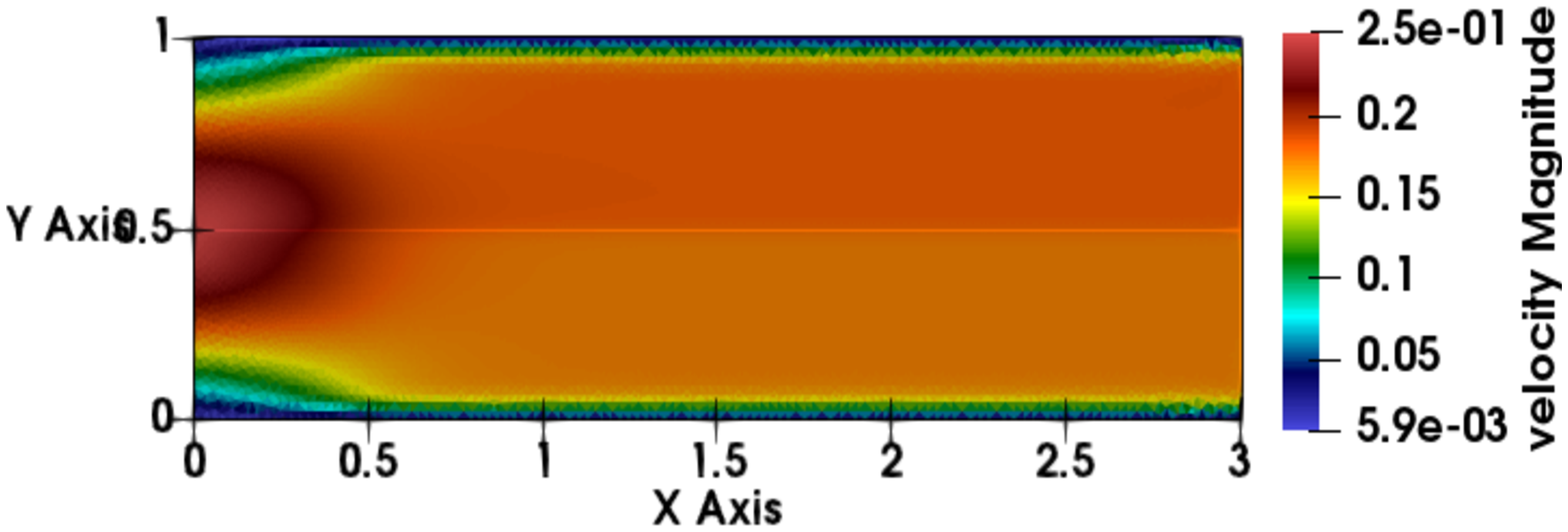}
		{  \text{(b) $t=0.083$ [s]}}
	\end{minipage}
	\begin{minipage}[b]{0.49\linewidth}
		\centering
		\includegraphics[width=1.\linewidth]{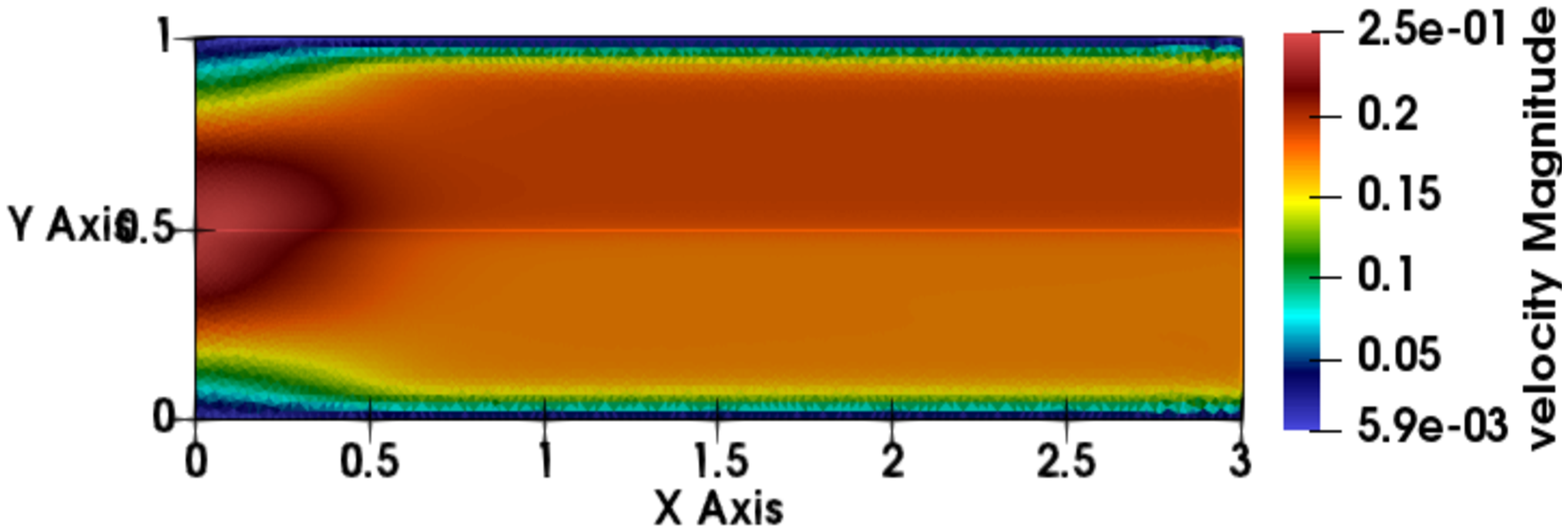}
		{  \text{(c) $t=0.16$ [s]}}
	\end{minipage}
	\begin{minipage}[b]{0.49\linewidth}
		\centering
		\includegraphics[width=1.\linewidth]{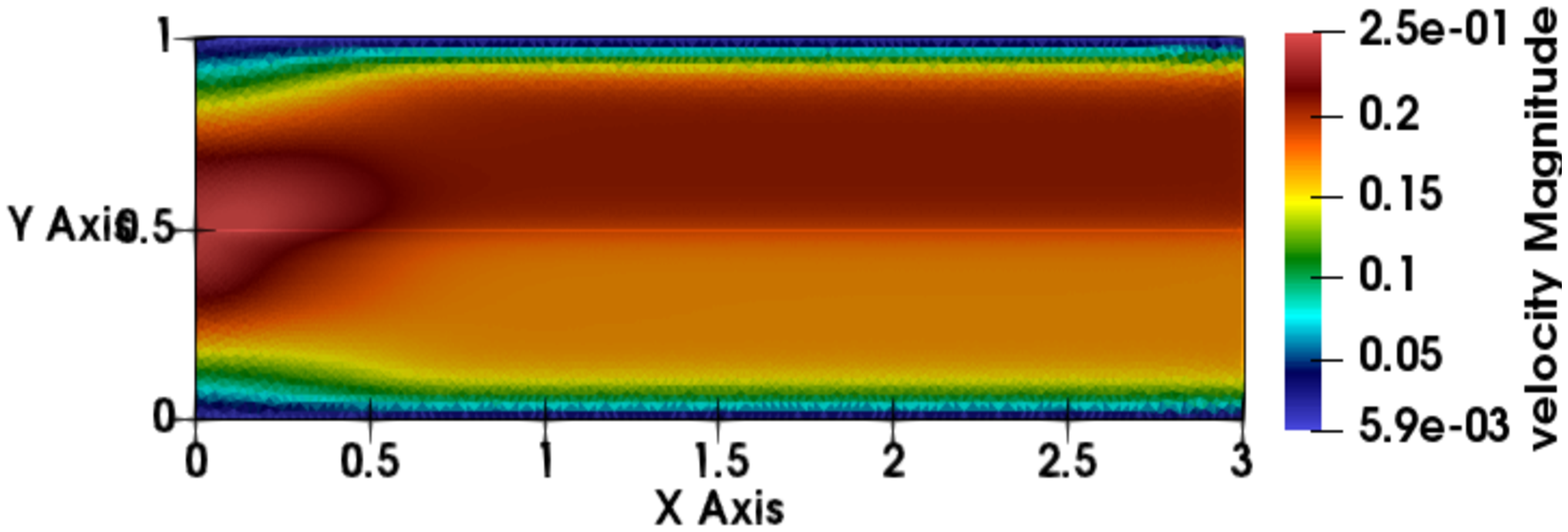}
		{  \text{(d) $t=0.33$ [s]}}
	\end{minipage}
	\begin{minipage}[b]{0.49\linewidth}
		\centering
		\includegraphics[width=1.\linewidth]{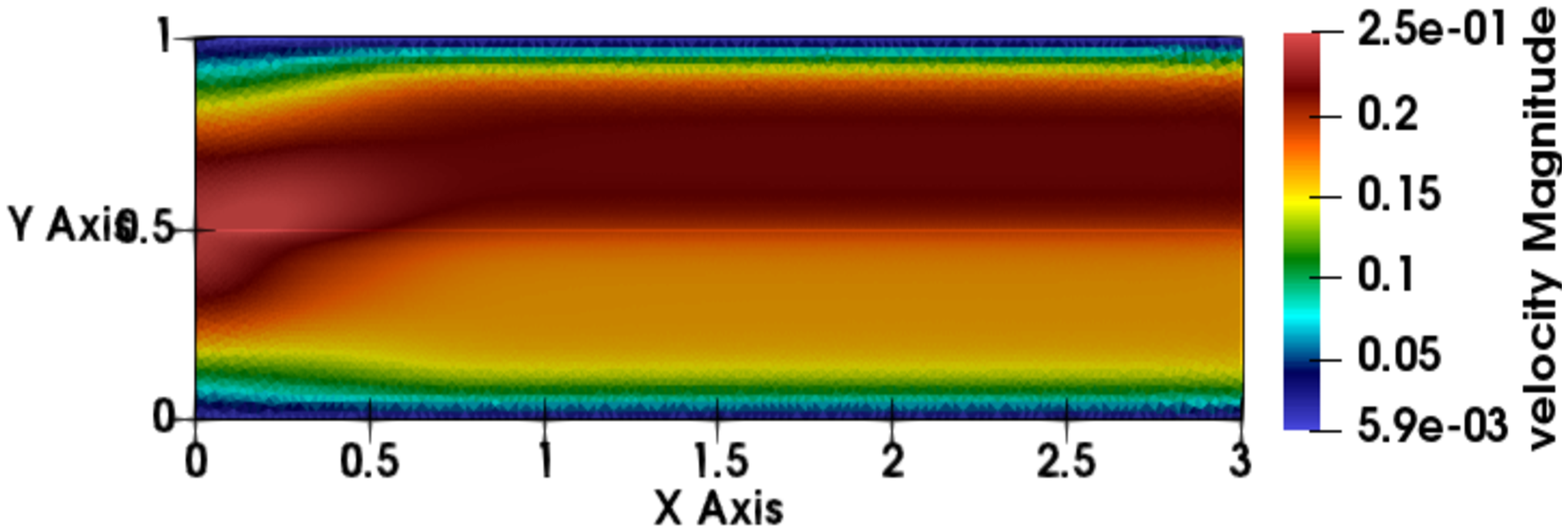}
		{  \text{(e) $t=0.5$ [s]}}
	\end{minipage}
	\begin{minipage}[b]{0.49\linewidth}
		\centering
		\includegraphics[width=1.\linewidth]{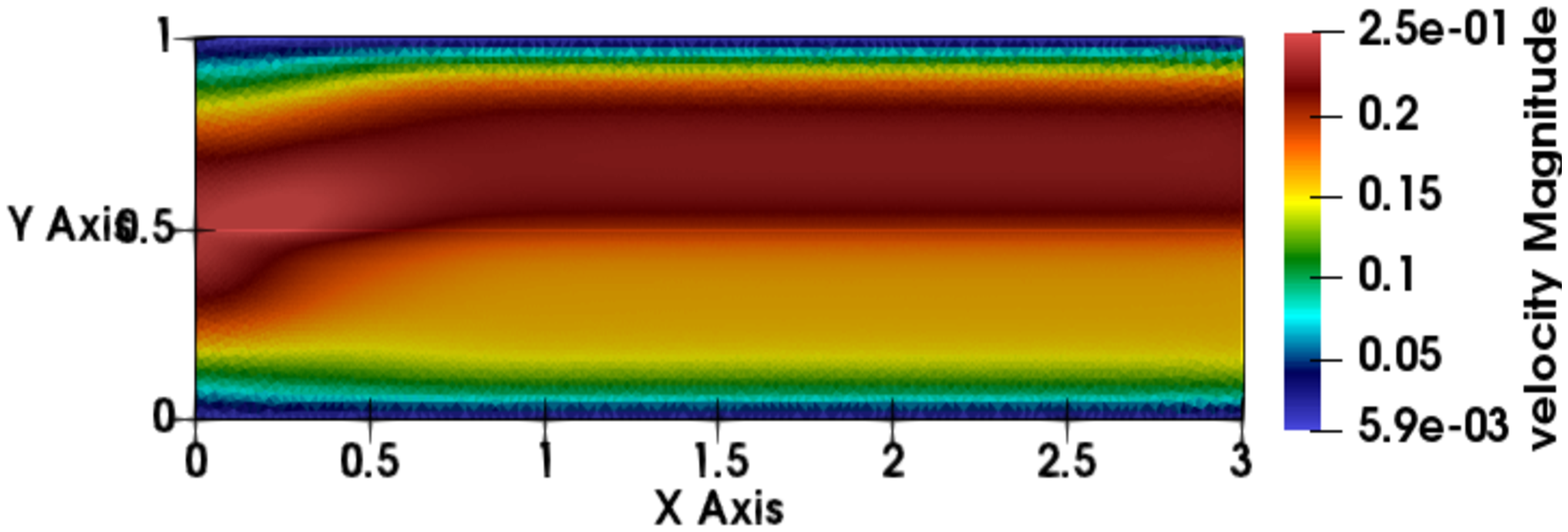}
		{  \text{(f) $t=0.66$ [s]}}
	\end{minipage}
	\begin{minipage}[b]{0.49\linewidth}
		\centering
		\includegraphics[width=1.\linewidth]{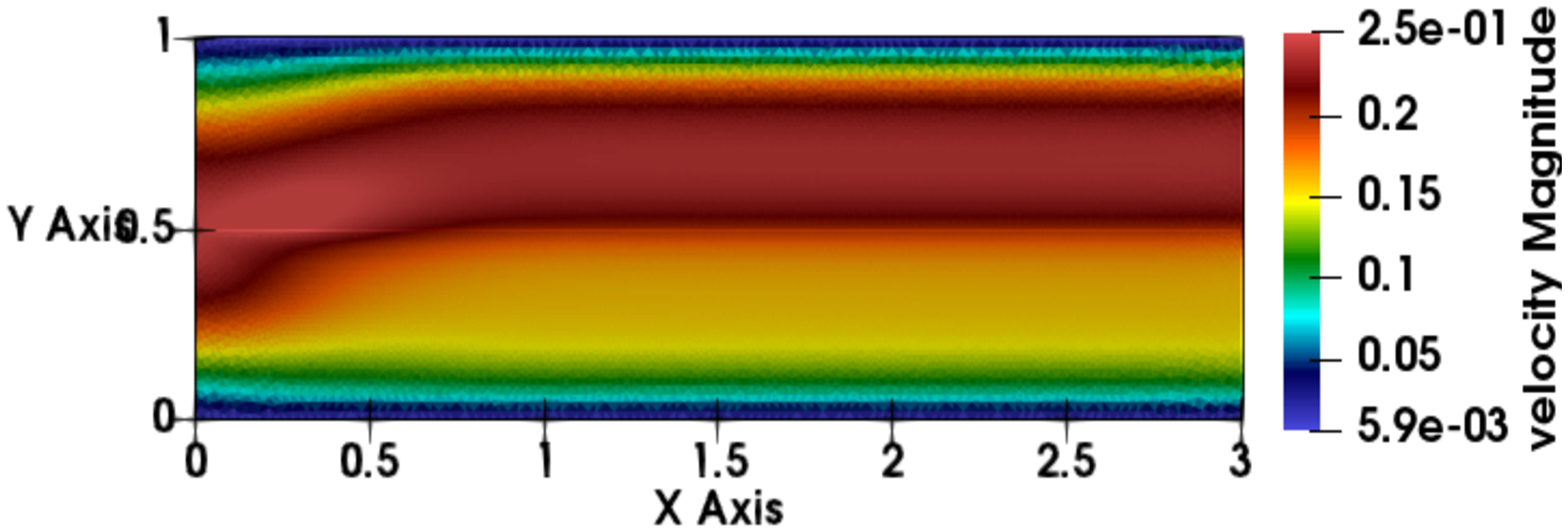}
		{  \text{(g) $t=0.83$ [s]}}
	\end{minipage}
	\begin{minipage}[b]{0.49\linewidth}
		\centering
		\includegraphics[width=1.\linewidth]{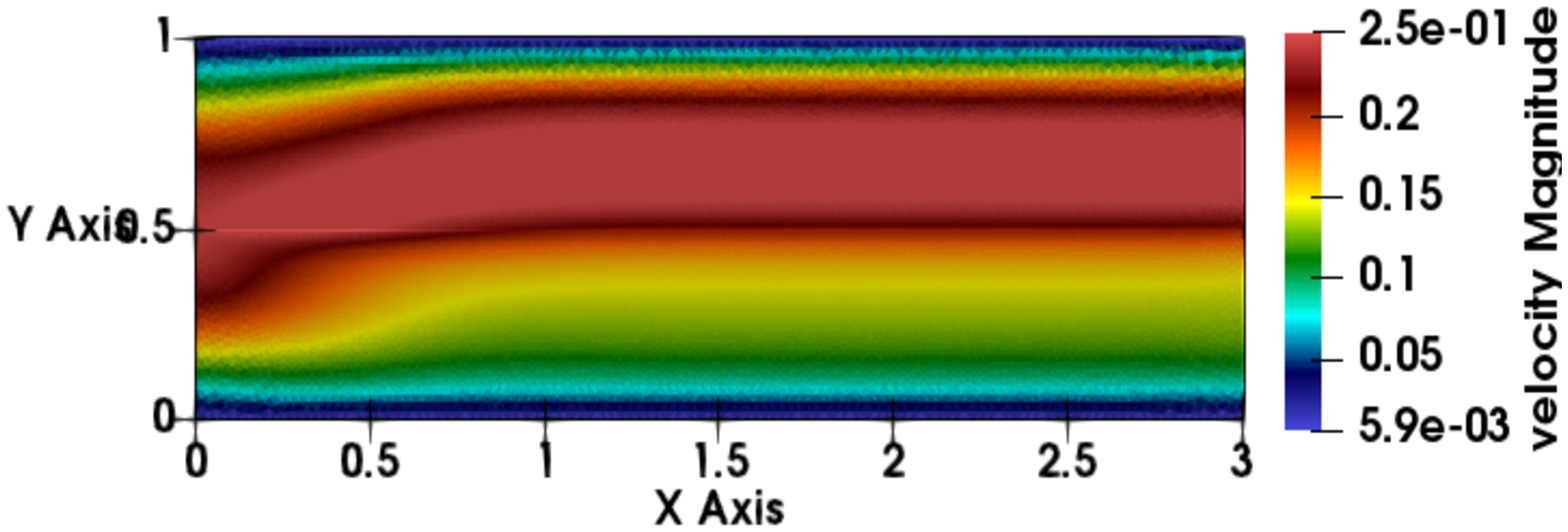}
		{  \text{(h) $t=1.6$ [s]}}
	\end{minipage}
	\begin{minipage}[b]{0.49\linewidth}
		\centering
		\includegraphics[width=1.\linewidth]{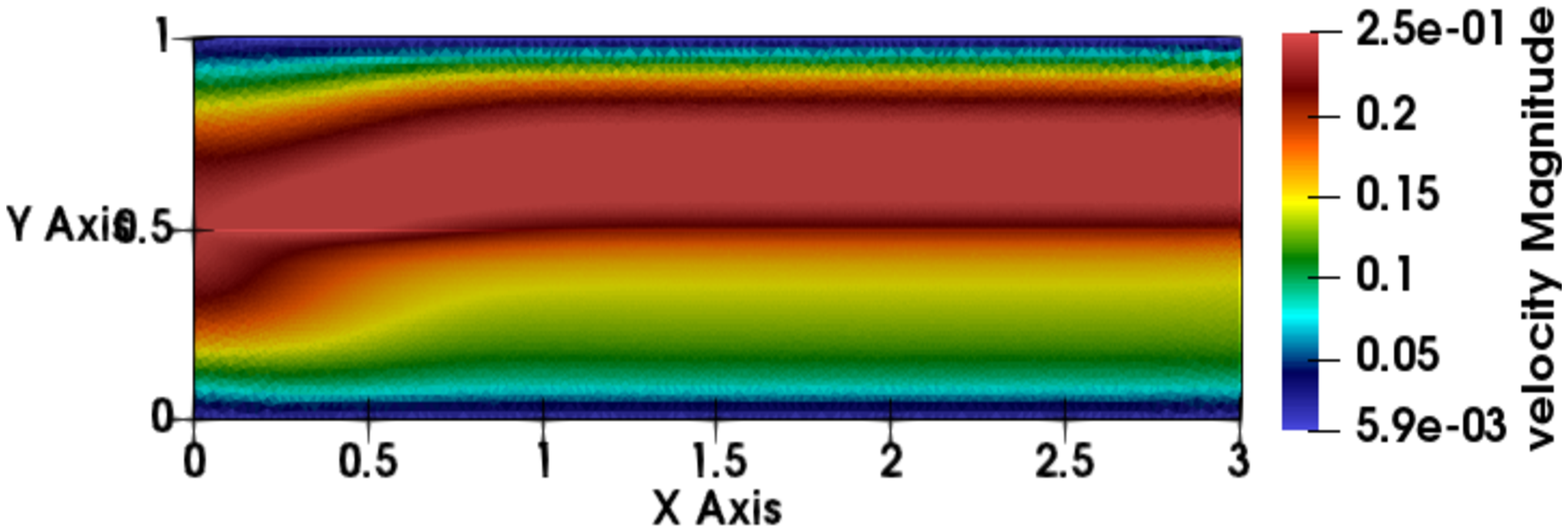}
		{  \text{(i) $t=3.3$ [s]}}
	\end{minipage}
	\begin{minipage}[b]{0.49\linewidth}
		\centering
		\includegraphics[width=1.\linewidth]{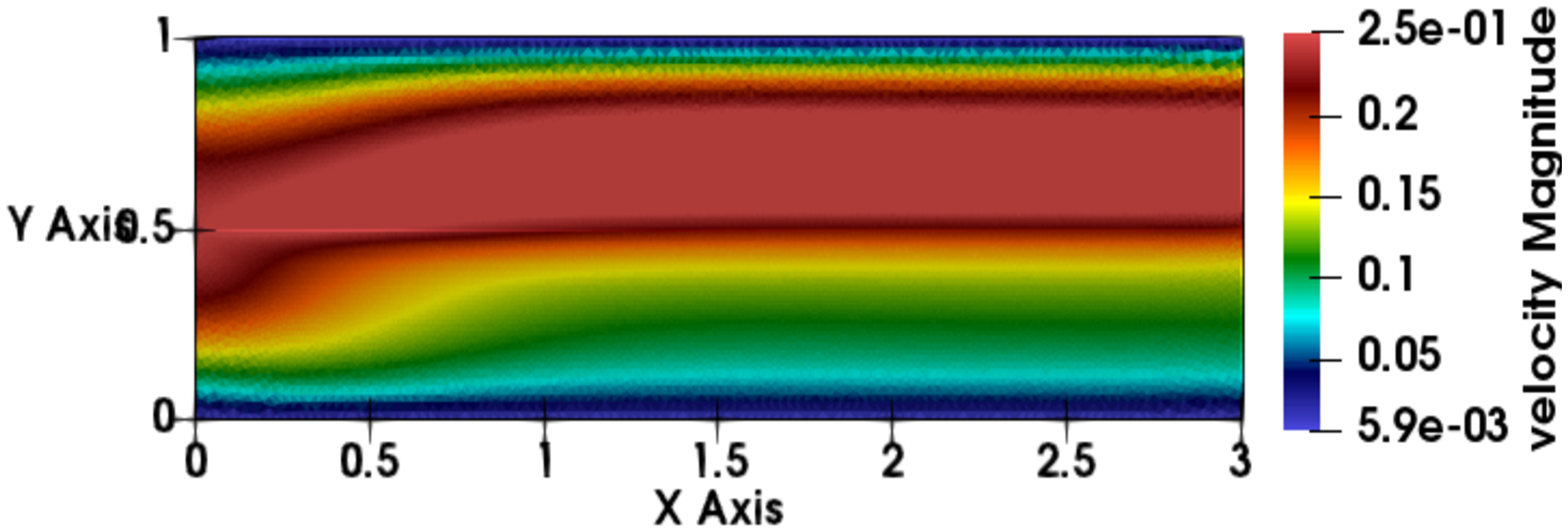}
		{  \text{(j) $t=5.0$ [s]}}
	\end{minipage}
	\caption{Time evolution of velocity magnitude.}
	\label{fig:ex1}
\end{figure}
\par
The results of the first case simulation are presented in Figure~\ref{fig:ex1}. Figure~\ref{fig:ex1}-(a) is the initial condition of the simulation. From this figure, we can see the profile distribution of velocity is symmetric. As long as the time increasing, the profile distribution becomes asymmetric; this happens because of the difference of values of the porosity. From equation~\eqref{def:F_and_K}, it can be understood that high porosity implies high permeability. High permeability means the resistance of fluids to flow is small so that the fluid can flow faster rather than the area with small porosity.
It clearly can be seen in (c)-(j) in Figure~\ref{fig:ex1}, the flow in the top layer with $\phi=0.8$ is faster than that in the bottom layer with $\phi=0.4$.
This behavior of our numerical results has a good agreement with the natural flow in the simple case of the porous media qualitatively.  
\par
The purpose of the second simulation is to understand the fluid flow in the complex value of porosity.
This simulation is motivated by the real condition of the porosity distribution in the rock structure, such as in carbonate rock, where the value of porosity is irregular.
For this simulation, we set $\Omega=(0,3\pi)\times(0,\pi)$~[cm], $T=5$~[s], $\rho=9.951\times 10^{-1}$~[gr/cm$^{3}]$, $\mu=8.89\times 10^{-3}$~[dyn.s/cm$^2]$, $f=0$, and
\[
u^{0} = \eta(x_{1})
\begin{pmatrix}
0.01\left( \frac{\pi^{2}}{4}-\left( x_2-\frac{\pi}{2}\right)^{2} \right) \\
0
\end{pmatrix},
\]
where $\eta$ is the function defined in~\eqref{eq:eta}.
For the porosity $\phi$ we set
\[
\phi (x) = \frac{\gamma_1 - \gamma_0}{2}\sin(2x_{2})\cos(2x_{1})+\frac{\gamma_1 + \gamma_0}{2},
\]
where $\gamma_0 = 0.15$ and $\gamma_1 = 0.65$.
For this case we run the simulation with division number $N=300$, $h=3\pi/N$, $\tau=h$.
To aid the understanding of problem setting in this simulation, we plotted the distribution function of porosity in the computational domain in Figure~\ref{fig:ex2_porosity}.  
\begin{figure}[htbp]
	\centering
	\includegraphics[width=3.0in]{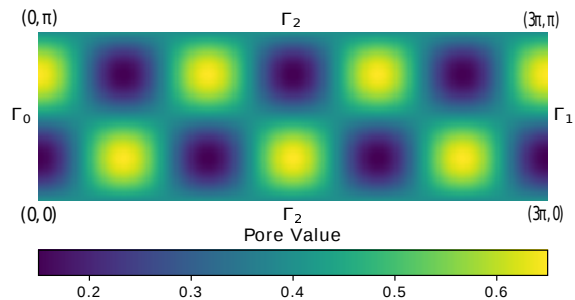}
	\caption{Computation domain and porosity value distribution}\label{fig:ex2_porosity}
\end{figure}
\begin{figure}[htbp]
	\begin{minipage}[b]{0.49\linewidth}
		\centering
		\includegraphics[width=1.\linewidth]{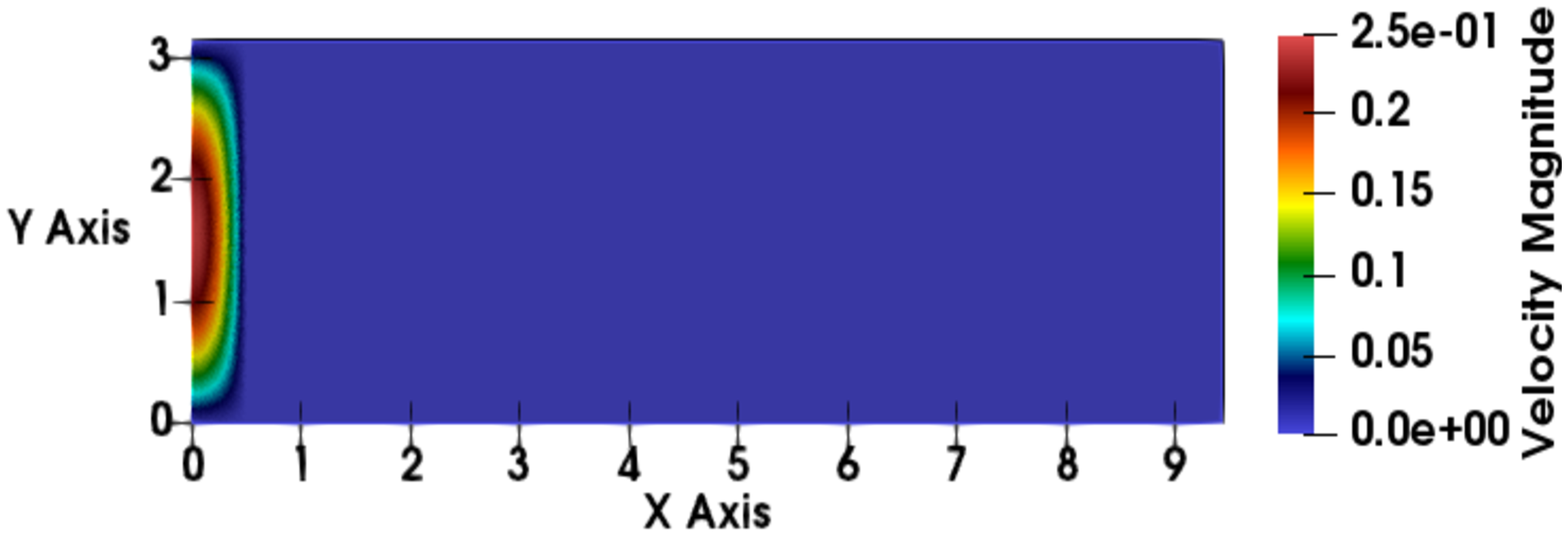}
		{  \text{(a) $t=0.0$ [s]}}
	\end{minipage}
	\begin{minipage}[b]{0.49\linewidth}
		\centering
		\includegraphics[width=1.\linewidth]{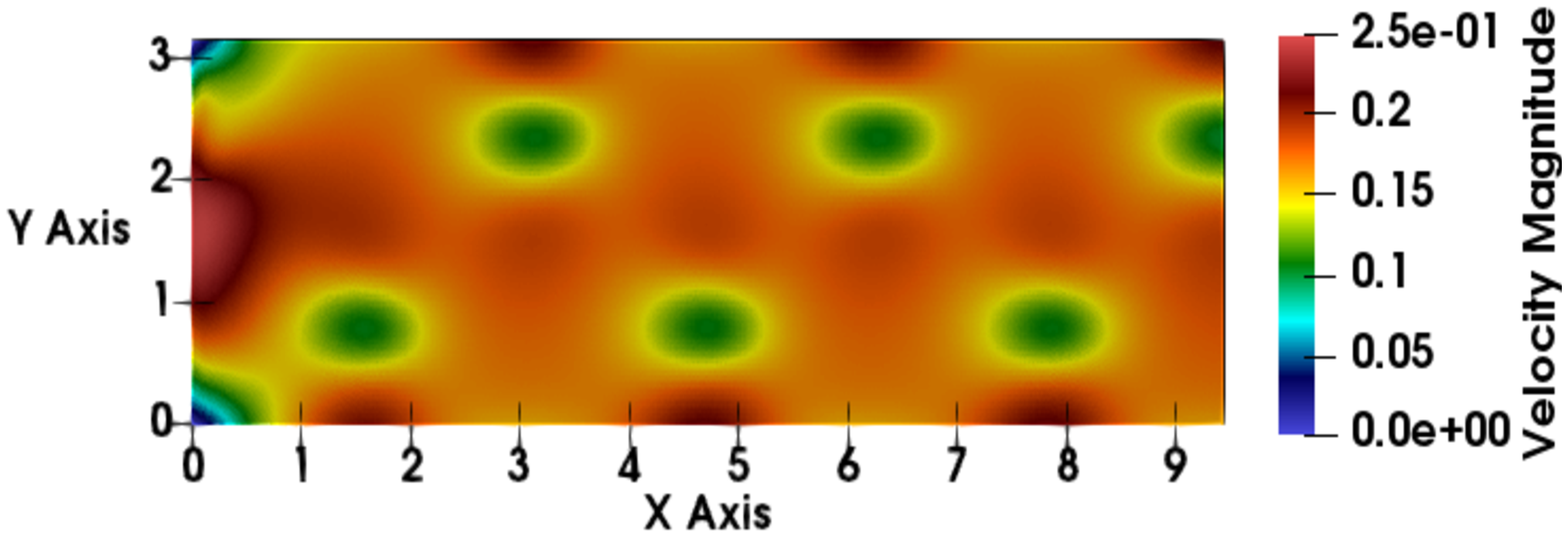}
		{  \text{(b) $t=0.083$ [s]}}
	\end{minipage}
	\begin{minipage}[b]{0.49\linewidth}
		\centering
		\includegraphics[width=1.\linewidth]{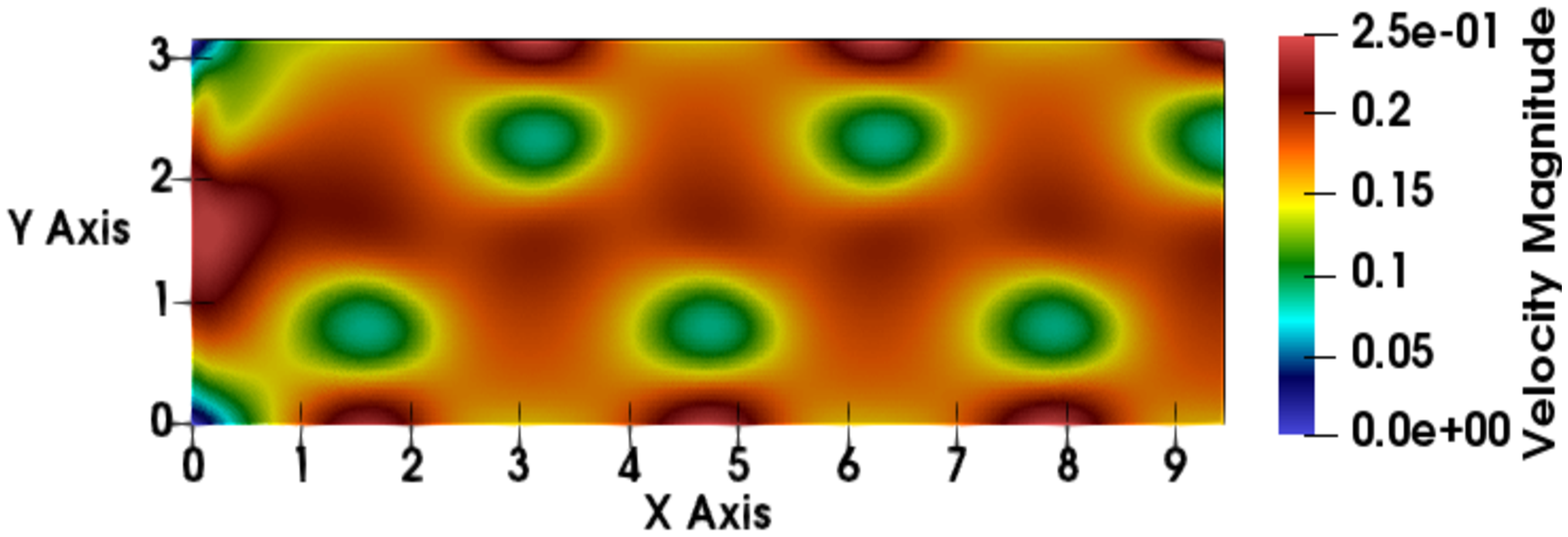}
		{  \text{(c) $t=0.16$ [s]}}
	\end{minipage}
	\begin{minipage}[b]{0.49\linewidth}
		\centering
		\includegraphics[width=1.\linewidth]{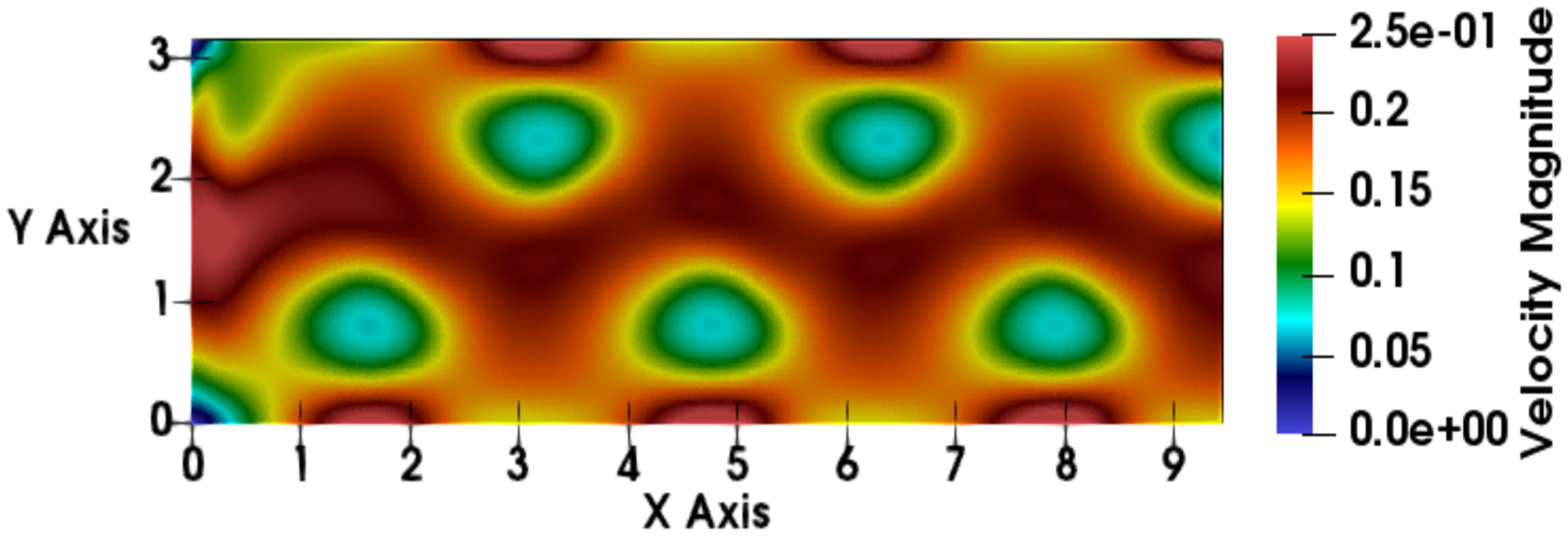}
		{  \text{(d) $t=0.33$ [s]}}
	\end{minipage}
	\begin{minipage}[b]{0.49\linewidth}
		\centering
		\includegraphics[width=1.\linewidth]{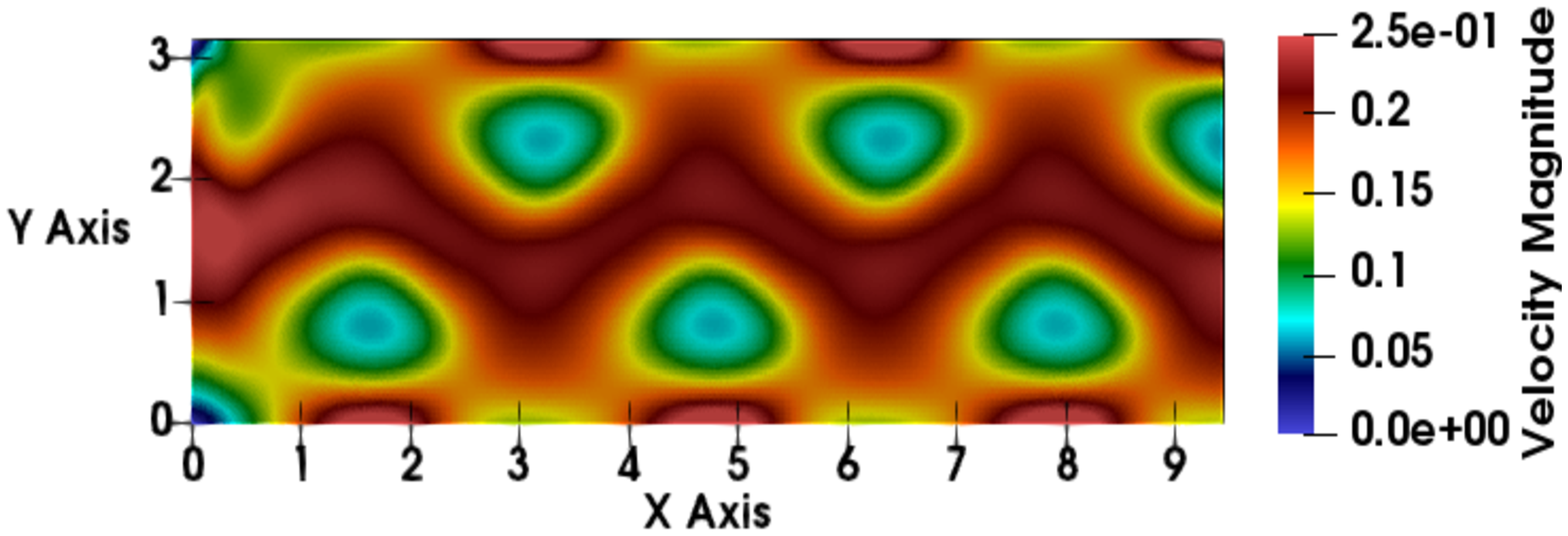}
		{  \text{(e) $t=0.5$ [s]}}
	\end{minipage}
	\begin{minipage}[b]{0.49\linewidth}
		\centering
		\includegraphics[width=1.\linewidth]{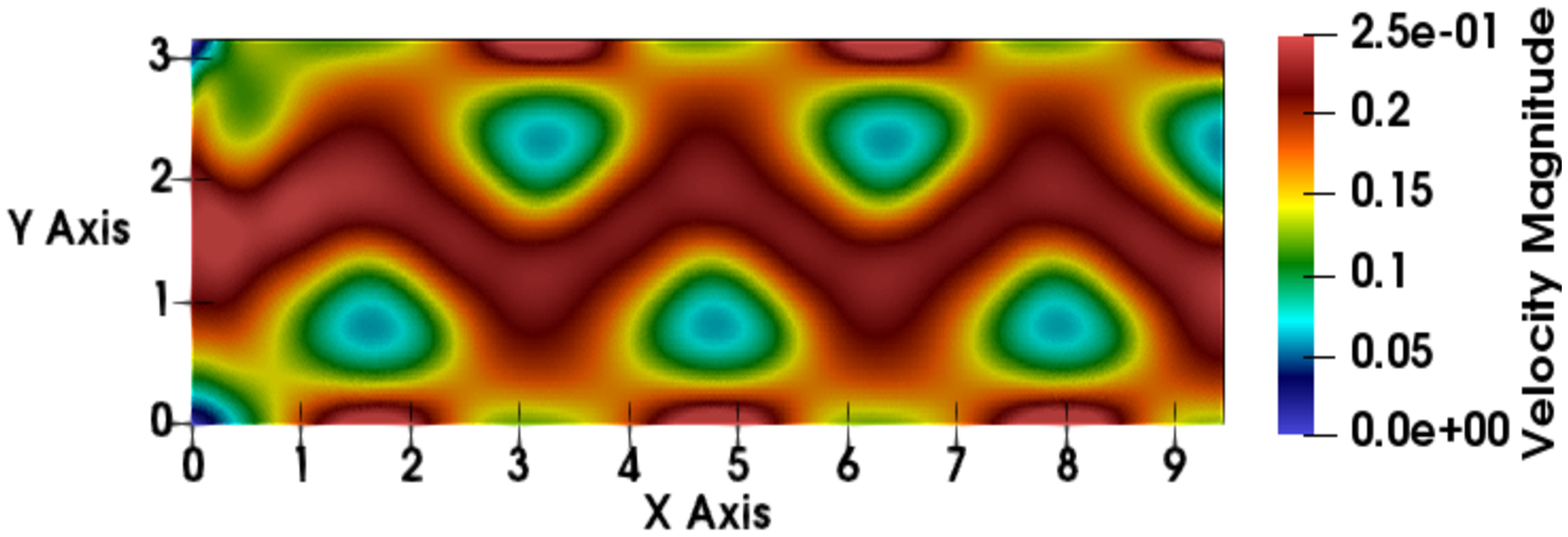}
		{  \text{(f) $t=0.66$ [s]}}
	\end{minipage}
	\begin{minipage}[b]{0.49\linewidth}
		\centering
		\includegraphics[width=1.\linewidth]{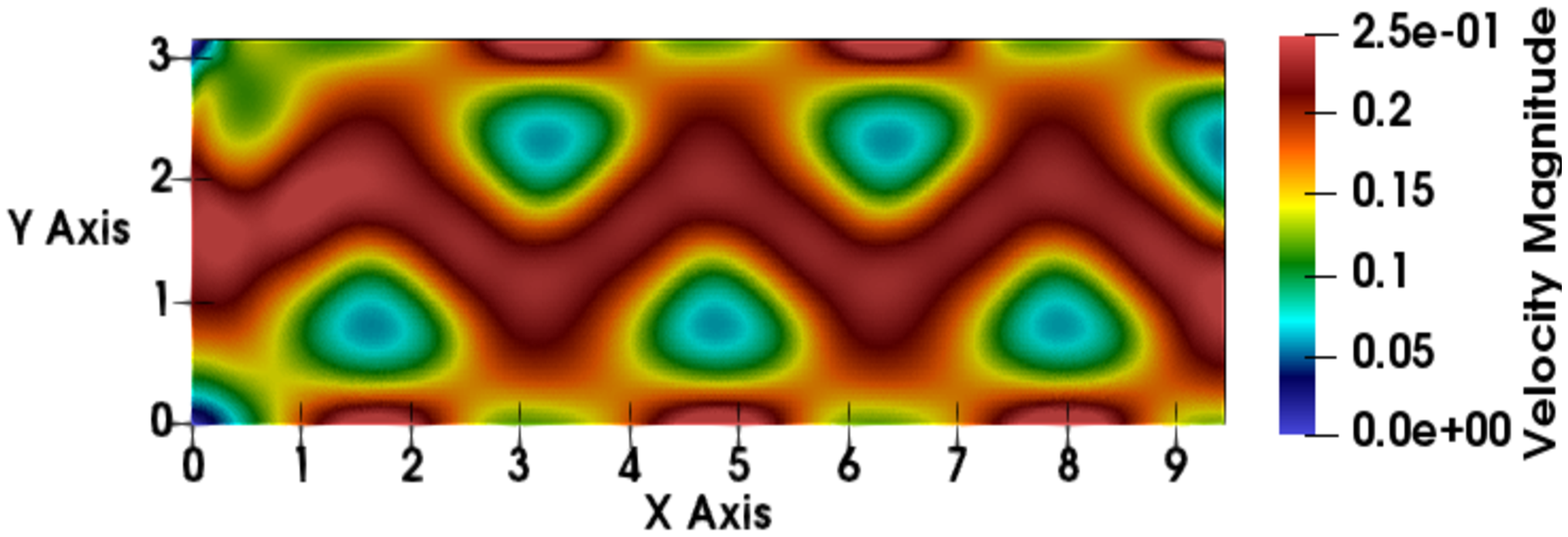}
		{  \text{(g) $t=0.83$ [s]}}
	\end{minipage}
	\begin{minipage}[b]{0.49\linewidth}
		\centering
		\includegraphics[width=1.\linewidth]{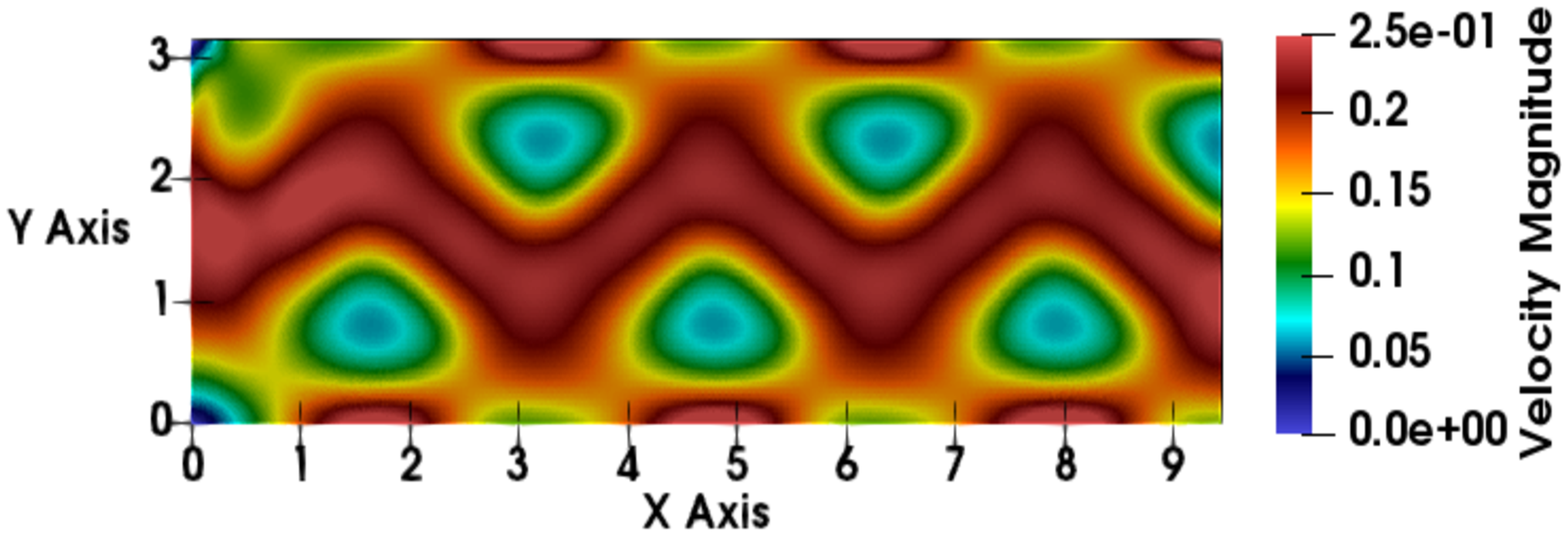}
		{  \text{(h) $t=1.6$ [s]}}
	\end{minipage}
	\begin{minipage}[b]{0.49\linewidth}
		\centering
		\includegraphics[width=1.\linewidth]{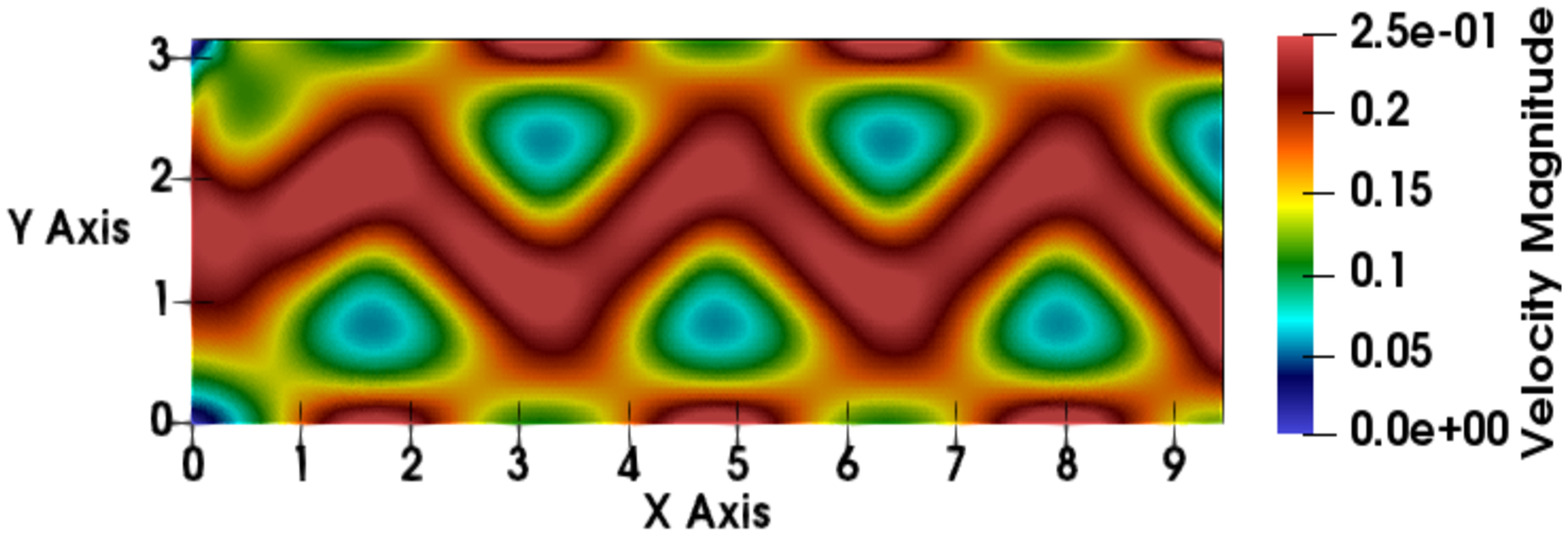}
		{  \text{(i) $t=3.3$ [s]}}
	\end{minipage}
	\begin{minipage}[b]{0.49\linewidth}
		\centering
		\includegraphics[width=1.\linewidth]{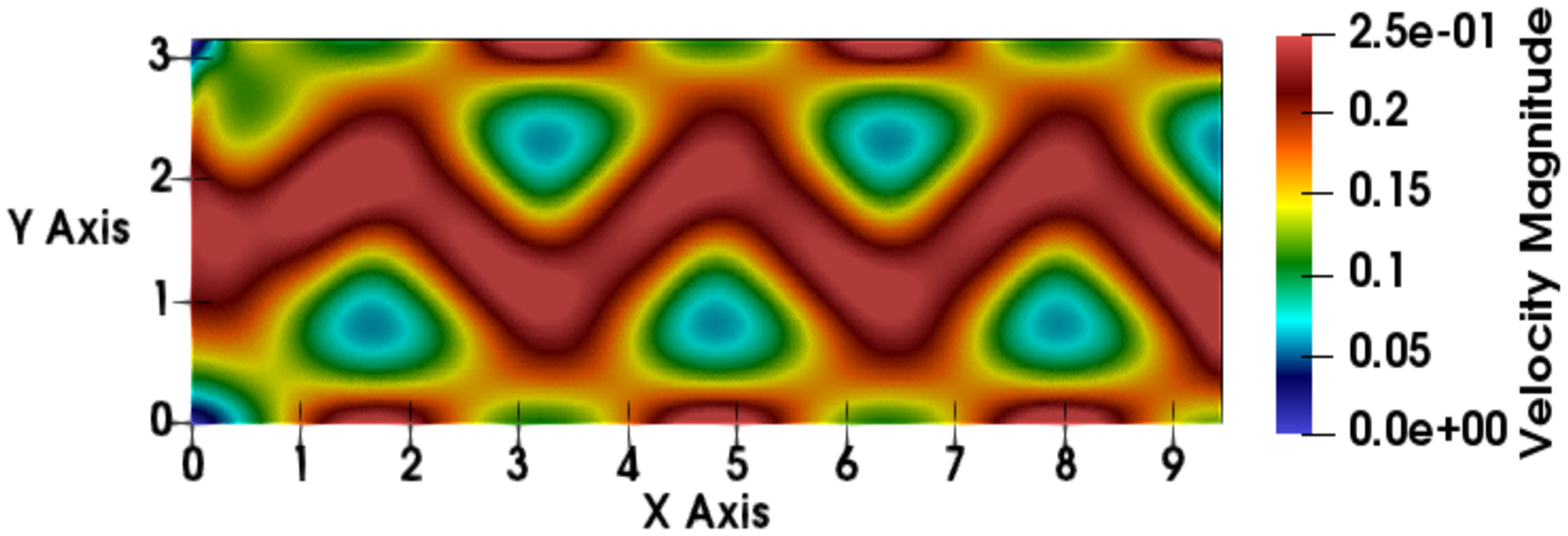}
		{  \text{(j) $t=5.0$ [s]}}
	\end{minipage}
	\caption{Time evolution of magnitude velocity.}
	\label{fig:ex2}
\end{figure}
\par
The results of the second case simulation are presented in Figure~\ref{fig:ex2}.
Figure~\ref{fig:ex2}-(a) is the initial velocity magnitude of the simulation.
From Figure~\ref{fig:ex2}, we can see that the fluid is flowing faster in the area which has a large porosity; for the area which has small porosity, the fluid is flowing slowly.
In the area which has small porosity, we can see the gradation motion of the fluid clearly; this fact emphasizes us that scheme~\eqref{scheme} can deal with the irregular pattern of porosity. 
Figure~\ref{fig:ex2} has a good agreement with the natural flow in the irregular design of porous media qualitatively.
\section{Conclusion}
We have proved the $L^{2}$-stability for the model proposed by Hsu and Cheng for fluid flow through porous media, where the non-Darcy drag force $-\rho F(\phi)\phi|u|u/\sqrt{K(\phi)}$ played an essential role. We also have introduced a new Lagrange--Galerkin scheme with the Adams--Bashforth method for solving that model numerically. Our new numerical scheme has second-order accuracy both in space and in time.
From the numerical simulation presented in Subsection~\ref{subsec:nonhomo_porosity} we have seen that the results have a good agreement with the natural flow in the simple and irregular cases of the porous media qualitatively.  
\section*{Acknowledgments}
This work is partially supported by
MEXT (Ministry of Education, Culture, Sports, Science, and Technology) scholarship,
JSPS KAKENHI Grant Number JP18H01135,
JSPS A3 Foresight Program,
and JST PRESTO Grant Number~JPMJPR16EA.


\begin{thebibliography}{99}

\bibitem{10}
M.J.~Ahammad and J.M.~Alam,
\newblock A numerical study of two-phase miscible flow through porous media with a Lagrangian model, \emph{The Journal of Computational Multiphase Flows}, \textbf{9} (2017), 127--143.
%
\bibitem{BouMadMetRaz-1997}
K.~Boukir, Y.~Maday, B.~M\'etivet, and E.~Razafindrakoto,
\newblock A high-order characteristics/finite element method for the incompressible {N}avier--{S}tokes equations, \emph{International Journal for Numerical Methods in Fluids}, \textbf{25} (1997), 1421--1454.
%
\bibitem{21}
S.C.~Brenner and L.R.~Scott,
\newblock \emph{{T}he {M}athematical {T}heory of {F}inite {E}lement {M}ethods, 3rd Edition},
\newblock Springer, New York, 2008.
%
\bibitem{27}
\newblock H.C. Brinkman, 
\newblock A calculation of the viscous force exerted by a flowing fluid on a dense swarm of particle,
\emph{Applied Scientific Research}, \textbf{1} (1947), 27--34.
%
\bibitem{18}
P.G.~Ciarlet,
\newblock \emph{{T}he {F}inite {E}lement {M}ethod for {E}lliptic {P}roblems},
\newblock North-Holland, Amsterdam, 1978.
%
\bibitem{1} 
\newblock F.~Cimolin and M.~Discacciati, 
\newblock Navier--Stokes/Forchheimer models for filtration through porous media, \emph{Applied Numerical Mathematics}, \textbf{72} (2013), 205--224.
%
\bibitem{6} 
\newblock M.~Choi, G.~Son and W.~Shim, 
\newblock A level-set method for droplet impact and penetration into a porous medium, \emph{Computers \& fluids}, \textbf{145} (2017), 153--166.
%
\bibitem{28}
\newblock D.M.~Dolberg, J.~Helgesen, T.H.~Hanssen, I.~Magnus, G.~Saigal, and B.K.~Pedersen , 
\newblock Porosity prediction from seismic inversion, Lavrans Field, Halten Terrace, Norway,
\emph{The Leading Edge},\textbf{19(4)} (2000), 392--399.
%
\bibitem{9} 
\newblock S.~Ergun, 
\newblock Fluid flow through packed columns, \emph{Chemical Engineering Progress}, \textbf{48} (1952), 89--94.
%
\bibitem{EwiRus-1981}
R.E.~Ewing and T.F.~Russell,
\newblock Multistep Galerkin methods along characteristics for convection-diffusion problems,
\newblock In Vichnevetsky, R. and Stepleman, R.S. editors, \textit{Advances in Computer Methods for Partial Differential Equations},
IMACS, \textbf{IV}~(1981), 28--36.
%
\bibitem{19}
V.~Girault and P.-A.~Raviart, 
\newblock \emph{{F}inite {E}lement {M}ethods for {N}avier--{S}tokes {E}quations, {T}heory and {A}lgorithms},
\newblock Springer, Berlin, 1986.
%
\bibitem{16}
\newblock A.~Hazen, 
\newblock Some Physical Properties of Sand and Gravels with Special Reference to Their Use in Filtration,
\emph{24th Annual Report, Massachusetts State Board of Health} (1893), 539--556.
%
\bibitem{26}
\newblock F.~Hecht,
\newblock New development in FreeFem++,
{\em Journal of Numerical Mathematics}, \textbf{20}~(2012), 251--265.
%
\bibitem{7} 
\newblock C.T.~Hsu and P.~Cheng, 
\newblock Thermal dispersion in a porous medium,
\emph{International Journal of Heat and Mass Transfer}, 
\textbf{33} (1990), 1587--1597.
%
\bibitem{4} 
\newblock M.K.~Hubbert,
\newblock Darcy's law and the field equations of the flow of underground fluids, \emph{Hydrological Sciences Journal}, \textbf{2} (1957), 23--59.
%
\bibitem{13} 
\newblock M.R.~Islam, M.E.~Hossain, S.H.~Mousavizadegan, S.~Mustafiz, J.H.~Abour-Kassem,
\newblock \emph{Advance petroleum reservoir simulation, 2$^{nd}$ edition},
\newblock Scrivener, Canada, 2016.
%
\bibitem{2} 
\newblock G.A.~Nasilio, O.~Buzzi, S.~Fityus and T.S.~Yun, D.W.~Smith , 
\newblock Upscaling of Navier--Stokes equations in porous media: Theoretical, numerical, and experimental approach, \emph{Computers and Geotechnics}, \textbf{36} (2009), 1200--1206.
%
\bibitem{20}
J.~Ne\v{c}as,
\newblock \emph{{L}es {M}\'ethods {D}irectes en {T}h\'eories des {\'E}quations {E}lliptiques},
\newblock Masson, Paris, 1967.
%
\bibitem{11} 
\newblock D.A.~Nield,
\newblock The limitations of the Brinkman-Forchheimer equation in modeling flow in a saturated porous medium and at an interface, \emph{International Journal of Heat and Fluid Flow}, \textbf{12} (1991), 269--272.
%
\bibitem{5} 
\newblock D.A.~Nield, 
\newblock Modeling fluid flow and heat transfer in a saturated porous medium, \emph{Applied mathematics \& decision sciences.}, \textbf{4} (2000), 165--173.
%
\bibitem{14} 
\newblock  D.A.~Nield and A.~Bejan,
\newblock \emph{Convection in porous medium},
\newblock 5$^{th}$ edition,  Springer, Switzerland, 2016.
%
\bibitem{3} 
\newblock P.~Nithiarasu,~K.N. Seetharamu and T.~Sundararajan, 
\newblock Natural convection heat transfer in a fluid saturated variable porosity medium, \emph{International Journal of Heat and Mass Transfer}, \textbf{40} (1997), 3955--3967.
%
\bibitem{8} 
\newblock H.~Notsu and M.~Tabata, 
\newblock Error estimates of a stabilized Lagrange--Galerkin scheme for the Navier--Stokes equation, \emph{Mathematical modeling and numerical analysis.}, \textbf{50} (2016), 361--380.
%
\bibitem{22}
\newblock H.~Notsu and M.~Tabata,
\newblock Error estimates of a stabilized Lagrange--Galerkin scheme of second-order in time for the Navier--Stokes equations,
\newblock In Y. Shibata and Y. Suzuki (eds.), \emph{Mathematical Fluid Dynamics, Present and Future}, 497--530, Springer, 2016.
%
\bibitem{23}
\newblock H.~Notsu and M.~Tabata,
\newblock Stabilized Lagrange--Galerkin schemes of first- and second-order in time for the Navier--Stokes equations,
\newblock In Y. Bazilevs and K. Takizawa (eds.), \emph{Advances in Computational Fluid-Structure Interaction and Flow Simulation: New Methods and Challenging Computations}, 331--343, Springer, 2016.
%
\bibitem{17} 
\newblock W.~Sobieski, A.~Trykozko, 
\newblock Darcy's and Forchheimer's laws in practice. Part 1. The experiment, \emph{Technical Sciences}, \textbf{17(4)} (2014), 321--335.
%
\bibitem{15} 
\newblock Y.~Su, J.H.~Davidson,
\newblock \emph{Modeling approaches to natural convection in porous medium},
\newblock SpringerBriefs in Applied Sciences and Technology, Springer, New York, 2015.
%
\bibitem{24}
\newblock H.~Teng and T.S.~Zhao,
\newblock An extension of Darcy's law to non-Stokes flow in porous media,
\emph{Chemical Engineering Science}, \textbf{55} (2000), 2727--2735.
%
\bibitem{12} 
\newblock S.~Whitaker, 
\newblock The transport equations for multi-phase systems,
\emph{Chemical Engineering Science}, \textbf{28} (1973), 139--147.
%
\bibitem{25}
\newblock L.~Wang, L.-P.~Wang, Z.~Guo and J.~Mi,
\newblock Volume-average macroscopic equation for fluid flow in moving porous media,
\emph{International Journal of Heat and Mass Transfer}, \textbf{82} (2015), 357--368.
%
\end{thebibliography}
\end{document}